\documentclass[a4paper,reqno,oneside,12pt,final]{amsart} 

\usepackage{fullpage}
\usepackage{amsmath}
\usepackage{amsfonts}
\usepackage{amssymb}
\usepackage{verbatim,enumitem,graphics}
\usepackage{xcolor}
\usepackage[notcite,notref,color]{showkeys}  
\usepackage[colorlinks,bookmarks,linkcolor=black,citecolor=black,urlcolor=black]{hyperref} 
\usepackage{dsfont}
\usepackage{tikz}
\usetikzlibrary{backgrounds, graphs, graphs.standard, quotes, 3d, patterns, calc, fit, shapes.geometric, decorations.pathreplacing, fadings, patterns, math}

\usepackage[utf8]{inputenc}
\usepackage{graphicx}

\usepackage{thm-restate}

\newtheorem{theorem}{Theorem}[section]
\newtheorem{lemma}[theorem]{Lemma}

\newtheorem{corollary}[theorem]{Corollary}

\newtheorem{question}[theorem]{Question}
\newtheorem{definition}[theorem]{Definition}

\newtheorem{claim}[theorem]{Claim}

\newcommand{\vass}[1]{\left|#1\right|}
\newcommand{\llb}{\left\lbrace}
\newcommand{\rrb}{\right\rbrace}

\newcommand{\cC}{\mathcal{C}}
\newcommand{\cH}{\mathcal{H}}
\newcommand{\cG}{\mathcal{G}}

\newcommand{\oldqed}{}
\def\endofClaim{\hfill\scalebox{.6}{$\Box$}}

\newcommand{\Int}{\mathrm{int}}

\newenvironment{claimproof}[1][Proof]{
  \renewcommand{\oldqed}{\qedsymbol}
  \renewcommand{\qedsymbol}{\endofClaim}
  \begin{proof}[#1]
}{
  \end{proof}
  \renewcommand{\qedsymbol}{\oldqed}
}

\setlist{itemsep=2pt,parsep=1pt,topsep=3pt,partopsep=0pt}  
\setenumerate{leftmargin=*,labelindent=\parindent} 

\title{Graphs with large minimum degree and no small odd cycles are $3$-colourable}

\author[J. B\"ottcher]{Julia B\"ottcher}
\address{(JB) London School of Economics, Department of Mathematics, Houghton Street, London WC2A 2AE, UK}
\email{j.boettcher@lse.ac.uk}

\author[N. Frankl]{N\'{o}ra Frankl}
\address{(NF) School of Mathematics and Statistics, The Open University, UK, and Alfr\'ed R\'enyi Institute of Mathematics, Budapest, Hungary}
\email{nora.frankl@open.ac.uk}

\author[D. Mergoni Cecchelli]{Domenico Mergoni Cecchelli}
\address{(DMC) London School of Economics, Department of Mathematics, Houghton Street, London WC2A 2AE, UK}
\email{d.mergoni@lse.ac.uk}

\author[O. Parzcyk]{Olaf Parczyk}
\address{(OP) Freie Universit\"at Berlin, Department of Mathematics and Computer Science, Arnimallee 3, 14195 Berlin, Germany}
\email{parczyk@mi.fu-berlin.de}

\author[J. Skokan]{Jozef Skokan}
\address{(JS) London School of Economics, Department of Mathematics, Houghton Street, London WC2A 2AE, UK}
\email{j.skokan@lse.ac.uk}

\thanks{NF was supported by an LMS Early Career Fellowship at an earlier stage of this work, and was Partially supported by ERC
Advanced Grant ``GeoScape''. OP was funded by the Deutsche Forschungsgemeinschaft (DFG, German Research Foundation) under Germany’s Excellence Strategy – The Berlin Mathematics Research Center MATH+ (EXC-2046/1, project ID: 390685689).}

\date{\today}

\begin{document}

\begin{abstract}
  Answering a question by Letzter and Snyder, we prove that for large enough $k$
  any $n$-vertex graph~$G$ with minimum degree at least $\frac{1}{2k-1}n$ and
  without odd cycles of length less than $2k+1$ is $3$-colourable.  In fact, we
  prove a stronger result that works with a slightly smaller minimum degree.
\end{abstract}

\maketitle

\thispagestyle{empty}

\section{Introduction}
Determining the chromatic number of a graph is a difficult problem. This explains why a
wealth of results in graph theory aims instead at determining meaningful upper bounds on
this quantity, which can also be seen as bounds on the permitted
structural complexity of the graphs under consideration. One natural question in
this direction then is if the chromatic number of the family of $\cH$-free
graphs is bounded for finite non-trivial~$\cH$. Here, a graph~$G$ is \emph{$\cH$-free} for a set~$\cH$ of
graphs, if~$G$ does not contain any member of $\cH$ as a subgraph, and~$\cH$ is \emph{non-trivial} 
if none of the graphs in~$\cH$ is a forest. This question
was answered negatively by Erdős~\cite{Erd:largeGirth} in one of the early applications of the celebrated
probabilistic method: For every finite non-trivial~$\cH$ and
every positive integer~$c$, there are $\cH$-free graphs with chromatic number at least~$c$.

In another influential paper, Erdős and Simonovits~\cite{Erdos} asked
what happens if a minimum degree condition is also introduced. More
precisely, they initiated the study of the so-called chromatic profile
of~$\cH$. To define this, it is convenient to introduce some notation.  We
denote by $\cG(\cH)$ the family of all $\cH$-free graphs, and by
$\cG(\cH,\alpha)$ the set of graphs in $\cG(\cH)$ of minimum degree at least
$\alpha|V(G)|$. For $c\geq 2$ a positive integer, the \emph{chromatic profile}
of $\cH$ as a function in~$c$ is
\[\delta_{\chi}(\cH, c)=\inf \{\alpha\in [0,1]: \forall G\in \cG(\cH,\alpha),\  \chi(G)\leq c\}.\]
This function measures how large the minimum degree needs to be in order to guarantee that an
$\cH$-free graph has chromatic number at most $c$. Erdős and Simonovits~\cite{Erdos} judged that in full generality this quantity seemed `too complicated'
to study. Despite considerable progress in the last few decades, this judgment
still stands firm. The goal of this paper is to contribute to the understanding
of the chromatic profile of the family of odd cycles up to a certain length.

But let us first summarise what is known. Soon after Erdős and Simonovits's paper,
Andrásfai, Erdős and Sós \cite{Andrasfai} proved that $\{K_{r}\}$-free graphs of
minimum degree strictly larger than $\frac{3r-7}{3r-4}|V(G)|$ have chromatic
number at most $r-1$. Moreover, in the same paper examples were given of
$\{K_{r}\}$-free graphs of minimum degree $\frac{3r-7}{3r-4}|V(G)|$ whose
chromatic number is exactly $r$. In other words,
$\delta_{\chi}(\{K_r\},r-1)=\frac{3r-7}{3r-4}$. Other known results include
$\delta_{\chi}(\{K_3\},3)=\frac{10}{29}$ by Haggkvist \cite{Haggkvist} and Jin
\cite{Jin}, and $\delta_{\chi}(\{K_3\},c)=\frac{1}{3}$ for every $c\geq 4$ by
Brandt and Thomassé~\cite{Brandt}. On the other hand, Thomassen~\cite{Thomassen} 
showed that $\delta_{\chi}(\{C_5\},c)\le\frac{6}{c}$ and more generally an upper bound for
$\delta_{\chi}(\{C_k\},c)$; together with a result of Ma~\cite{Ma2016} this implies for every fixed~$k$ that
\begin{equation*}
\label{eq:TomMa}
\Omega\Big((k+1)^{-4(c+1)}\Big) = \delta_{\chi}\big(\{C_k\},c\big)= O\Big(\frac{k}{c}\Big)\,.
\end{equation*}

The next developments concerned two quantities related to
the chromatic profile: the chromatic threshold and the homomorphism threshold.
The \emph{chromatic threshold} of~$\cH$ is
\[\delta_{\chi}(\cH)=\inf \{\alpha\in [0,1]: \exists K \textrm{ s.t. }  \forall G\in \cG(\cH,\alpha),\ \chi(G)\leq K\},\] 
and measures how large the minimum degree needs to be to guarantee that the
chromatic number of $\cH$-free graphs is bounded by some constant.  For example,
the result by Brandt and Thomassé~\cite{Brandt} mentioned above shows that
$\delta_{\chi}(\{C_3\})=\frac13$.  The chromatic threshold is by now
much better understood than the chromatic profile.  Building on the work of
{\L}uczak and Thomass\'e~\cite{LucTho}, and generalising various previous
results, Allen, B\"{o}ttcher, Griffiths, Kohayakawa, and Morris~\cite{Allen}
determined the chromatic threshold of every finite family $\cH$. For more
details about the history of the study of the chromatic threshold see this paper
and the references therein.

To get an even better picture, one can consider the more restrictive notion of
the {\it homomorphism threshold} $\delta_{\hom}(\cH)$ of a family $\cH$, which
is a measure of the smallest minimum degree that guarantees that $\cH$-free
graphs are homomorphic to a small $\cH$-free graph. That is,
\[\delta_{\hom}(\cH)=\inf \{\alpha \in [0,1]: \exists F\in \cG(\cH) \textrm{ s.t. }\forall G\in \cG(\cH,\alpha),\ G \textrm{ is hom.\ to } F\}.\] 
Note that $\delta_{\hom}(\cH) \ge \delta_{\chi}(\cH)$. Determining homomorphism
thresholds is distinctively harder than determining chromatic thresholds.
{\L}uczak~\cite{Luczak} showed that that for~$K_3$ the homomorphism threshold
equals the chromatic threshold, which as discussed above is $\frac13$.  Goddard
and Lyle~\cite{GodLyl} and Nikiforov~\cite{Niki} extended this to all cliques,
showing $\delta_{\hom}(\{K_k\})=\delta_{\chi}(\{K_k\})=\frac{2k-5}{2k-3}$.
Letzter and Snyder~\cite{Letzter} considered a generalisation to longer odd
cycles instead. They proved $\delta_{\hom}(\{C_5\})\le\tfrac{1}{5}$ and
$\delta_{\hom}(\cC_5)=\tfrac{1}{5}$, where
$\cC_{2k-1}=\{C_3,\dots,C_{2k-1}\}$ is the family of odd cycles up to length
$2k-1$. Extending this, Ebsen and Schacht~\cite{Ebsen} proved
$\delta_{\hom}(\{C_{2k-1}\})\le\tfrac{1}{2k-1}$ and 
$\delta_{\hom}(\cC_{2k-1})=\tfrac{1}{2k-1}$ for all $k\ge 2$. Complementing the
first of these results, Sankar~\cite{Sankar} recently proved hat
$\delta_{\hom}(\{C_{2k-1}\})>0$ for all $k\ge 2$.
This shows that, in contrast to cliques, the homomorphism threshold for odd
cycles behaves differently than the chromatic threshold since
$\delta_{\chi}(\cC_{2k-1})=0$ for $k>2$.

Returning to the chromatic profile, what can be said about families of odd
cycles?  Already the methods by Andrásfai, Erdős and Sós~\cite{Andrasfai} give
$\delta_{\chi}(\cC_{2k-1},2) = \tfrac{2}{2k+1}$, where the lower bounds comes
from a blow-up of $C_{2k+1}$.  Moving on to $3$-colourability, when establishing
the homomorphism threshold for~$\cC_5$, Letzter and Snyder~\cite{Letzter} showed
that graphs in $\cG(\cC_5,\frac{1}{5}+\varepsilon)$ are in fact homomorphic to
graphs of chromatic number $3$, which implies $\delta_{\chi}(\cC_5,3)\leq
\frac{1}{5}$.  The best-known lower bound, on the other hand, is
$\delta_{\chi}(\cC_5,3)\geq\frac{14}{73}$, which is given by an asymmetric
blow-up of a $\cC_5$-free graph on $22$ vertices (cf.\ the graph~$G_{3,3}$ in
Van Ngoc and Tuza~\cite{Tuza}).  The homomorphisms constructed in Ebsen and
Schacht's~\cite{Ebsen} generalisation, however, were not to $3$-colourable
graphs. Thus, their result does not imply an upper bound on
$\delta_{\chi}(\cC_{2k-1},3)$.

Providing such an upper bound is the main contribution of this paper. 
We prove that for $k$ large
enough the homomorphism threshold $\frac{1}{2k-1}$ is an upper bound on the
chromatic profile $\delta_{\chi}(\cC_{2k-1},3)$.  This answers a question of
Letzter and Snyder~\cite{Letzter}. In fact, we can prove a slightly stronger
upper bound, which shows that $\delta_{\chi}(\cC_{2k-1},3)$ is strictly smaller
than $\delta_{\hom}(\cC_{2k-1})$ for large~$k$.

\begin{theorem}
\label{thm:main}
For any $t \in \mathbb{N}$ and any integer $k \ge k(t) = 5490 + 45 t$ the following holds.
Any $\cC_{2k-1}$-free graph $G$ of minimum degree at least $\tfrac {1}{2k+t}|V(G)|$ is $3$-colourable.
In other words, for $k\geq k(t)$, we have $\delta_{\chi}(\cC_{2k-1},3)\leq \frac{1}{2k+t}$.
\end{theorem}

Since $k(t)$ is linear in $t$, we conclude
that there exists an $\varepsilon>0$ such that $\delta_{\chi}(\cC_{2k-1},3)\leq
\frac{1}{(2+\varepsilon)k}$ for large enough $k$.  Concerning lower bounds for
$\delta_{\chi}(\cC_{2k-1},3)$, we only know of bounds that are much
smaller. Such a bound can be achieved for example as follows. Take a generalised
$\cC_{2k-1}$-free Mycielski graphs of minimum degree $3$, and chromatic number
$4$, as described for example in Van Ngoc and Tuz~\cite{Tuza}; a balanced blow
up of these constructions gives $4$-chromatic $\cC_{2k-1}$-free graphs of
minimum degree $\frac{3}{2k^2+k+1}|V(G)|$.  As this lower bound and our upper
leave a considerable gap, we make no further effort here in optimising the
constant factor in either of them.

It would be interesting to know if our upper bound
or this lower bound provides the right order of magnitude for
$\delta_{\chi}(\cC_{2k-1},3)$.

\begin{question}
  Is $\delta_{\chi}(\cC_{2k-1},3)\ge\frac{c}{f(k)}$ with~$c$ constant and~$f(k)$ linear in~$k$?
  Or is $\delta_{\chi}(\cC_{2k-1},3)\le\frac{c}{f(k)}$ with~$c$ constant and~$f(k)$ quadratic in~$k$?
\end{question}

Similarly, we did not try to optimise our~$k(t)$, since with our proof technique
one probably cannot bring this down to a single digit when $t=0$. Nevertheless, 
it would be interesting to know what happens for small~$k$. In particular,
our result motivates the following question.

\begin{question}
  Is $\delta_{\chi}(\cC_{5},3)< \tfrac 15 = \delta_{\hom}(\cC_5)$?
\end{question}

Finally, we remark that with our argument it is easy to derive a more general upper bound $\delta_\chi(\cC_{2k-1},c) \le \tfrac{1}{2k\lfloor c/3 \rfloor}$ for $c \ge 3$ and sufficiently large $k$.
Observe that $\delta_\chi(\cC_{2k-1},c)\le\delta_\chi(\{C_{2k-1}\},c)$ and thus this upper bound complements the bound $\delta_{\chi}\big(\{C_k\},c\big)= O\Big(\frac{k}{c}\Big)$ mentioned earlier in that it applies to the case of fixed~$c$ and large~$k$, while the latter is meaningful for fixed~$k$ and large~$c$.
We will briefly explain how our general bound can be obtained at the end of Section~\ref{sec:proof}.

\subsection*{Organisation}
    
The remainder of this paper is organised as follows. We start in
Section~\ref{sec:overview} with introducing some basic notation, explaining the
strategy of our proof of Theorem~\ref{thm:main}, providing the setup used in
this proof as well as the main lemmas we need for this, and outlining what
further will be needed for the proof of these lemmas. In Section~\ref{sec:proof}
we then prove Theorem~\ref{thm:main}. The proof of our main technical lemma
(Lemma~\ref{LemmaMainWeight345_actual}) is provided in
Section~\ref{SectionActualProof}. To prepare for this proof, we develop tools
for finding bipartite subgraphs in a weighted graph in
Section~\ref{GoodSection}, and for lower bounding the neighbourhood size of
certain cycles and paths in Section~\ref{BipartiteComplement}. We will explain
as part of Section~\ref{sec:overview} how these tools are used.


\section{Notation and overview of the proof}
\label{sec:overview}

Before we explain the proof idea for our main theorem we review some (mostly)
standard notation and transfer it in a natural way to graphs that are equipped
with weights on their edges.

\subsection*{Notation}
Let $G$ be a graph and let $B\subseteq V(G)$ be a set of vertices.  We denote by
$G[B]$ the subgraph of~$G$ induced by~$B$. If $G[B]$ is connected, then we also
say as a shorthand that~$B$ is \emph{connected}.
We write $G\setminus B$ for the graph $G[V(G)\setminus B]$.
As usual, $N(v)$ denotes the (open) \emph{neighbourhood} of a vertex~$v$ of~$G$. 
For a set of vertices $B$, we denote with
$\Int(B)=\llb v\in B: N(v)\subseteq B\rrb$ the \emph{interior} of $B$.
We write~$B^c$ for the \emph{complement} $V(G)\setminus B$ of~$B$ in~$G$.
For a graph $G$ the \emph{distance} $d_G(x,y)$ of two vertices $x,y$ in $G$ is the minimum number of edges of a path in $G$ with end-vertices $x$ and $y$.
For two sets of vertices $A,B \subseteq V(G)$ the \emph{distance} $d_G(A,B)$ is the minimum of $d_G(x,y)$ over all $x \in A$ and $y \in B$.
For an integer $i\ge 0$ the (closed) \emph{$i$-th neighbourhood} of $B$ in $G$ is given by
\[N^i_G[B]=\llb x\in V(D): \exists v\in B \textrm{ s.t. } d_D(x,v)\le i \rrb.\]
Often we also omit the subscript~$G$ when it is clear from the context in which
graph we are taking neighbourhoods. We remark that if $D$ is a subgraph of
$G$ on a smaller vertex set we also write $N^i[D]$ instead of
$N^i\big[V(D)\big]$.

We shall also work with the following type of auxiliary graphs with weights on their edges.
For a graph $H$ a \emph{weight function} is a function of the form $\omega : E(H) \rightarrow \mathbb{N}$ and a graph endowed with such a function is called a \emph{weighted graph}.
All concepts defined above for unweighted graphs also apply to weighted graphs.
The weight of a subgraph~$H'$ of~$H$ is $\omega(H')=\sum_{e\in E(H')}\omega(e)$.  We say
that~$H$ is \emph{weighted bipartite} if there is no cycle in~$H$ of odd weight.
HeWe also say that $B \subseteq V(H)$ is weighted bipartite when $H[B]$ is and the graph is clear from the context.

In addition to the notion of unweighted distance defined above, for weighted graphs we shall also use a weighted version as follows. The \emph{weighted distance} $d_{\omega,H}(x,y)$ of
two vertices $x,y$ in a weighted graph~$H$ is the minimum weight of a path from $x$ to $y$.
Moreover, for any vertex $v$ and for an integer $i\ge 0$, we define the
(closed) \emph{weighted $i$-th neighbourhood}
around $v$ as
\[N^i_\omega[v]=\llb x\in V(H): d_\omega(x,v)\le i\rrb.\]

\subsection*{Overview of the proof}
The starting point of our proof of Theorem~\ref{thm:main} is inspired by Thomassen's
approach~\cite{Thomassen} to establishing the chromatic threshold of~$C_5$.
As in that approach, we start by fixing a maximal set of non-adjacent vertices
$v_1,\dots,v_h$ with disjoint neighbourhoods $N(v_1),\dots,N(v_h)$, which leaves
a set of remaining vertices~$X=V(G)\setminus \bigcup_{i=1}^h N^1[v_i]$, and then
analyse the structure of our graph based on the resulting vertex partition.
However, our analysis uses different and new ideas and is substantially more complex as
we work with a different setup.

It turns out that given any two of the vertices above, say $v_i, v_j$, the crucial information we need for this analysis is the length of a shortest path between~$N(v_i)$
and~$N(v_j)$ whose internal vertices lie in~$X$. Moreover, we only care about this path if it is of length at most~$3$. Such a path of length at most~$3$ gives
a $v_i,v_j$-path of length in $\{3,4,5\}$.  Consequently, one main idea in our
proof is to represent the structure of our graph by introducing an auxiliary
weighted graph~$H$ on the vertex set $[h]$. In $H$ we have an edge~$ij$ whenever such a $v_i,v_j$-path with length in $\{3,4,5\}$ exists; moreover, we assign as a weight to the edge $ij$ the length of
the path between $v_i$ and $v_j$. Since our graph has no odd cycles
of length smaller than $2k+1$, it is easy to see that this auxiliary graph has no cycles of odd weight smaller than $2k+1$. Moreover, by assuming that $G$ is connected and by choosing the vertices $v_1,\dots,v_h$ carefully, we can guarantee that~$H$ has a spanning tree of edges of weight~$3$. This is the motivation for the definition of the following family of graphs.

\begin{definition}\label{MainSetting345}
  For $k \in \mathbb{N}$ we denote by $\mathcal{H}(k)$ the family of graphs $H$ with the following two properties:
  \begin{itemize}
      \item There is a weight function $\omega:E(H)\to \llb 3,4,5 \rrb$ on the edges of $H$ such that in $H$ there are no cycles $C$ such that $\omega(C)$ is odd and smaller than $2k+1$.
      \item There is a tree $T$ spanning $H$ such that all edges of~$T$ have weight $3$.
  \end{itemize}
  Furthermore, we denote by $\mathcal{H}(k,s)$ the graphs in $\mathcal{H}(k)$ on at most $s$ vertices.
\end{definition}

This auxiliary graph $H$ encapsulates substantial structural information of $G$, which is essential in our proof of the $3$-colourability of $G$.
We do the latter by obtaining a partition of the vertex set of $G$ as required by the following lemma.

\begin{lemma}
  \label{lem:bipartite}
  Let~$G$ be a graph on vertex set~$V$. If there is a set of vertices $A\subseteq V$
  such that $G[A]$ is connected, $G[V\setminus A]$ is bipartite, and for all $v\in V\setminus A$ we have that 
  $G[A\cup \{v\}]$ is bipartite, then $\chi(G) \le 3$.
\end{lemma}

\begin{proof}
  First choose a colouring of $A$ using colours $\{1,2\}$, and a colouring of
  $V\setminus A$ using colours $\{3,4\}$. Note that since $A$ is connected,
  every neighbour of a vertex $v\in V\setminus A$ in $A$ is of the same
  colour. Now we recolour vertices of colour $4$ as follows. If a vertex $v\in
  V\setminus A$ of colour $4$ is connected to a vertex of colour $1$, recolour
  it with $2$, otherwise recolour it with $1$.
\end{proof}

This criterion that guarantees $3$-colourability motivates the following decomposition lemma of the auxiliary graph $H$, which is the heart of our proof of Theorem~\ref{thm:main}.

\begin{restatable}[Main technical lemma]{lemma}{MainLemma}
  \label{LemmaMainWeight345_actual}
  For any $t \in \mathbb{N}$ and any integer $k \ge k(t) = 5490 + 45 t$ the following holds.
  For any $H \in \mathcal{H}(k,2k+t)$ there exists a
  subset $B$ of $V(H)$ such that $H[B]$ is connected, $H \setminus B$ is weighted bipartite,
  and $H[B \cup \{ v\}]$ is weighted bipartite for all $v \in V(H) \setminus B$.
\end{restatable}

Given this lemma the main task in proving Theorem~\ref{thm:main} is to ``translate'' this partition of the auxiliary graph $H$ into a partition of $G$ with essentially the same properties.
The proof of this lemma relies on a surgical analysis of the neighbourhood $N^1[C]$ of a cycle $C$ of odd weight and a careful combination of paths to build $B$. We now provide the key ideas of the argument together with additional lemmas in the next subsection, before turning to the proof of the main theorem.

\subsection*{Strategy to approach Lemma~\ref{LemmaMainWeight345_actual}}

We briefly discuss here the main ideas behind our proof of Lemma~\ref{LemmaMainWeight345_actual}, which is detailed in Sections~\ref{GoodSection},~\ref{BipartiteComplement} and~\ref{SectionActualProof}.
First, we note that we prove a statement which is slightly stronger than Lemma~\ref{LemmaMainWeight345_actual}. Indeed, when constructing the connected set $B$ we ensure that both $N[B]$ and $H\setminus B$ are weighted bipartite.

In Section~\ref{GoodSection} we study how to guarantee the first property. In particular, we are going to show that simple constructions like balls around a vertex and neighbourhoods of lightest paths are weighted bipartite. Moreover, we prove that when we select our sets carefully the property of being weighted bipartite passes to the union in a very precise way. This allows us to build larger weighted bipartite sets.

However, these results alone are not sufficient to obtain our goal. Indeed, once we get such a candidate set $B$, we need to prove that also $H\setminus B$ is weighted bipartite.
The following lemma shows that for this it is sufficient that the interior of $B$ is large enough.

\begin{lemma}\label{lemlargeInt}
  Let $k\geq 8$ and $t$ be natural numbers, and $H \in \mathcal{H}(k,2k+t)$.
  Let $B\subseteq V(H)$ be good and $\vass{\Int(B)}\ge\frac{4}{3}k+t$.
  Then $H \setminus B$ is weighted bipartite.
\end{lemma}

If the interior of $B$ has size at least $\frac{4}{3}k+t$, its complement is of size at most $|V(H)| - \vass{\Int(B)} \le \tfrac{2k}{3}$ and we are going to show that this is not enough space to contain a cycle of odd weight.
For example, it can not contain a cycle of odd weight with only edges of weight $3$, because this cycle has at least $\tfrac{2k+1}{3}$ vertices, too much to fit into the complement of $B$.
In general, the cycle might have edges of other weights, but the spanning tree of weight $3$ then guarantees that we find additional vertices in the neighbourhood of the cycle.
It turns out that also in general we get exactly the same bound as in the example above.

\begin{restatable}{lemma}{lemlargeweight}\label{lemlargeweight}
  Let $k\ge 8$ and $H \in \mathcal{H}(k)$.
  If $C$ is a non-spanning cycle of odd weight, then $\vass{N^1[C]}\geq \frac{2k+1}{3}$.
\end{restatable}

We would like to emphasise that exactly this is the reason why we require the spanning tree of weight $3$.
We quickly give the details of how to obtain Lemma~\ref{lemlargeInt} from Lemma~\ref{lemlargeweight}.

\begin{proof}[Proof of Lemma \ref{lemlargeInt}]
  By assumption on the interior,~$B$ is not empty.  Assume $H \setminus B$ is not
  weighted bipartite, and let $S\subseteq B^c=V(H) \setminus B$ be a cycle of odd in $H\setminus B$. As~$B$ is not empty this cycle is not spanning in~$H$.  Hence, we
  can apply Lemma~\ref{lemlargeweight} to conclude that $\vass{N^1[S]}\geq
  \tfrac{2k+1}{3}$. Since no vertex of $S\subseteq B^c$ can have a
  neighbour in $\Int(B)$, we have $N^1[S]\subseteq \Int(B)^c$ and so
  $\vass{\Int(B)^c}\geq \vass{N^1[S]}\ge\tfrac{2k+1}{3}$. However, we also have
  $\vass{V(H)}=2k+t$ and $\vass{\Int(B)}\ge \tfrac{4}{3}k+t$ which gives
  $\vass{\Int(B)^c}\le 2k-\tfrac{4}{3}k<\tfrac{2k+1}{3}$, a contradiction.  
\end{proof}

In Section~\ref{BipartiteComplement} we prove Lemma~\ref{lemlargeweight} and a useful corollary.
Finally, in Section~\ref{SectionActualProof} we combine the results presented in Sections~\ref{GoodSection} and~\ref{BipartiteComplement} to show the existence of a weighted bipartite set $B$ with large interior to prove Lemma~\ref{LemmaMainWeight345_actual}.
As promised we now turn to the proof of the main theorem.


\section{Proof of the main result}
\label{sec:proof}

\begin{proof}[Proof of Theorem \ref{thm:main}]
Let $t \in \mathbb{N}$ and let $k \ge k(t) = 5490 + 45 t$ be an integer.
Let $G=(V,E)$ be an $n$-vertex graph with minimum degree $\delta(G) \ge n/(2k+t)$ that does not contain an odd cycle of length shorter than $2k+1$.
Since we want to show that the chromatic number of $G$ is at most $3$, we may assume that $G$ is connected.

First, we construct an auxiliary graph $H$ on $h \le 2k+t$ vertices with weight function $w : E(H) \rightarrow \{3,4,5\}$ as follows.
Let $v_1 \in V$ be any vertex, set $V_1=\{ v_1 \} \cup N(v_1)$, and set the index~$i$ to $i=2$.
If possible, we pick a vertex $v_i \in V\setminus V_{i-1}$ such that $\{ v_i \} \cup N(v_i)$ is disjoint from $V_{i-1}$ and such that there is an edge between $N(v_i)$ and $N(v_j)$ for some $j$, $1 \le j \le i-1$. We let $V_{i} = \{ v_i \} \cup N(v_i)\cup V_{i-1}$, we increase the index~$i$ by one, and repeat the above. We stop this process when we cannot pick the vertex~$v_i$ anymore.
We let $h\ge 1$ be the index of the last vertex we picked before the process stopped. Note that $h \le 2k+t$ because $n \ge |V_h| > hn/(2k+t)$ by the minimum degree of~$G$. Furthermore, we set $X = V\setminus V_h$. 

Now let~$H$ be the graph with vertex set $[h]$ and with all edges $ij
\in \binom{[h]}{2}$ such that $d(N(v_i), N(v_j)) \le 3$. To every edge $ij\in E(H)$, we assign the weight
\[\omega(e)=d(N(v_i), N(v_j))+2\in \{ 3,4,5
\}\,,\] which is an upper bound on the distance between $v_i$ and $v_j$ in $G$.
We thus obtain a graph $H$ on $h \le 2k+t$ vertices and with weight function $w : E(H) \rightarrow \{ 3,4,5 \}$.
We observe the following simple properties of $H$ and the $V_i$.

\begin{claim}\label{cl:main:neighbour}\mbox{}
\begin{enumerate}[label={\rm ({P\arabic{*}})}]
    \item\label{main:P0} There is no cycle $C$ in $H$ whose weight is odd and less than $2k+1$.
    \item\label{main:P1} Each vertex $x\in X$ has a neighbour in $N(v_i)$ for some $i\in[h]$. 
    \item\label{main:P2} For every $i\in [h]$ the neighbourhood $N(v_i)$ is independent, if $k\geq 2$. 
    \item\label{main:P3} For every $i\in [h]$ the set $\{u\in V(H):d(u,v_i)=2\}$ is independent, if $k\geq 3$.
    \item\label{main:P4} If for some $i,j\in [h]$ there is a path of length $2$ from $N(v_i)$ to $N(v_j)$ in $G$, then $\omega(ij)=4$ in $H$, as long as $k\geq 4$.
    \item\label{main:P5} If for some $i,j\in [h]$ there is a path of length $3$ from $N(v_i)$ to $N(v_j)$ in $G$, then $\omega(ij)\in \{3,5\}$ in $H$, as long as $k\geq 5$.
\end{enumerate}
\end{claim}
\begin{claimproof}
  Property~\ref{main:P0} follows directly from our assumptions, because any cycle in $H$ of odd weight less than $2k+1$ would directly create an odd cycle of length less than $2k+1$ in $G$. Indeed, a cycle $C$ in $H$ with an odd weight less than $2k+1$ corresponds, by our construction, to a closed odd walk with less than $2k+1$ edges in~$G$, which in turn contains an odd cycle shorter than $2k+1$.

  To see~\ref{main:P1}, observe that if this was not the case then $\{x\}\cup N(x)$ would be disjoint from
  $V_{h}$. Hence, a shortest path from~$x$ to~$V_h$, which exists as~$G$ is
  connected, has length at least~$2$. But then the penultimate vertex on this path
  could be chosen as $v_{h+1}$, contradicting our assumption that the selection process
  stopped.

  Since an edge $pq$ in $N(v_i)$ gives a triangle $v_i,p,q$ in $G$ we
  obtain~\ref{main:P2}.  For~\ref{main:P3}, assume that $k\geq 3$ and there is
  an edge $pq$ in $\{u\in V(H):d(u,v_i)=2\}$. Let $p'$ be a neighbour of $p$ in $N(v_i)$,
  and $q'$ be a neighbour of $q$ in $N(v_i)$. Then $p,p',v_i,q',q$ is a closed
  walk of length $5$, a contradiction.

  Next we show~\ref{main:P4}.  Let $p,x,q$ be a path of length $2$ from $N(v_i)$
  to $N(v_j)$. Assume that $\omega(ij)\neq 4$. Then $\omega(ij)$ must be $3$, so
  there is an edge $p'q'$ between $N(v_i)$ and $N(v_j)$. But then
  $p,x,q,v_j,q',p',v_i$ is a closed walk of length $7$, a contradiction if
  $k\geq 4$.
  
  It remains to prove~\ref{main:P5}. Let $p,x,y,q$ be a path of length $3$ from
  $N(v_i)$ to $N(v_j)$, and assume that $\omega(ij)\notin \{3,5\}$. Then
  $\omega(ij)$ must be $4$, so there is a path $p',z,q'$ of length $2$ between
  $N(v_i)$ and $N(v_j)$. But then $p,x,y,q,v_j,q',z,p',v_i$ is a closed walk of
  length $9$, a contradiction if $k\geq 5$.
\end{claimproof}

It follows from the construction of $H$ that there is a spanning tree $T$ in $H$ with $\omega(e)=3$ for all $e \in E(T)$. As also Property~\ref{main:P0} holds, $H \in \mathcal{H}(k,2k+t)$.
As $k \ge k(t)$, by Lemma~\ref{LemmaMainWeight345_actual}, there exists a set $B \subseteq [h]$ such that $H[B]$ is connected, $H\big[[h] \setminus B\big]$ is weighted bipartite, and $H\big[B \cup \{ u \}\big]$ is weighted bipartite for all $u \in [h]$.

Our goal is to use this set $B$ to construct a set $A \subseteq V$ such that~$G$ and~$A$ satisfy the assumptions of Lemma~\ref{lem:bipartite}, so that we can conclude that~$G$ is $3$-colourable.
This is the case if~$A$ satisfies the following properties.
\begin{enumerate}[label={\rm ({A\arabic{*}})}]
    \item\label{main:first} $G[A]$ is connected,
    \item\label{main:second} $G[V\setminus A]$ is bipartite,
    \item\label{main:third} $G[A \cup \{ v\}]$ is bipartite for all $v \in V$.
\end{enumerate}

We construct $A$ as follows. Denote by $A_0$ the union of the sets $\{ v_b \}
\cup N(v_b)$ over all $b \in B$, let $X_0 \subseteq X=V \setminus V_h$ be the
set of vertices that have a neighbour in $A_0$, and set $A = A_0 \cup X_0$. It
remains to verify that $A$ satisfies conditions
\ref{main:first}--\ref{main:third}.

Since $H[B]$ is connected, we immediately
obtain~\ref{main:first}: Indeed, it is easy to verify that $G[A]$ is connected
if $H[B]$ is connected and if additionally for any edge $bb'$ in $H[B]$ we have a path
from~$v_b$ to~$v_{b'}$ in $G[A]$. The latter, however, is the case because by definition of~$H$ we
have $d\big(N(v_b), N(v_{b'})\big) \le 3$ and this can only hold if there is an
edge between $N(v_b)$ and $N(v_{b'})$ in~$G$.

For proving that~\ref{main:second} also holds, we shall use the following claim.

\begin{claim}\label{cl:main:neighbourB}
  Each vertex in $X \setminus X_0$ has a neighbour in
  some $N(v_i)$ with $i \in [h]\setminus B$.
\end{claim}
\begin{claimproof}
  Any vertex in $X$ has a neighbour in some $N(v_i)$ with
  $i \in [h]$ by~\ref{main:P1} of Claim~\ref{cl:main:neighbour}. In addition, $X_0 \subseteq X$
  contains all the vertices that have a neighbour in some $N(v_i)$ with $i
  \in B$. The claim follows.
\end{claimproof}

This allows us to show~\ref{main:second}.

\begin{claim}\label{claim5.2}
 $G[V\setminus A]$ is bipartite.
\end{claim}
\begin{claimproof}
  Assume that $G[V\setminus A]$ is not bipartite, and fix an odd cycle $C$ of
  shortest length. Recall that the set $V\setminus A$ consists of vertices in
  $\{v_i\}\cup N(v_i)$ with $i \in [h]\setminus B$, and the vertices in $X
  \setminus X_0$.
  
  We start with the following operations.  Removing from~$C$ all vertices in
  $C\cap \{v_1,\dots,v_h\}$ gives a collection $Q'_1,\dots,Q'_{\ell'}$ of pairwise
  vertex-disjoint paths, unless we are in the degenerate case that $C\cap
  \{v_1,\dots,v_h\}=\emptyset$ in which we simply let $Q'_1=C$.  Observe that by
  definition of~$A$ each removed vertex $v_j$ has $j\in[h]\setminus B$.  In each
  $Q'_i$ we now further identify all vertices in $\bigcup_{j \in [h]\setminus
    B}N(v_j)$ and split~$Q'_i$ along these vertices into (sub)paths. More
  precisely, for a fixed~$i$ let $Q'_i=q'_1,\dots,q'_{s'}$ and let
  $j_1\le\dots\le j_\ell$ be all indices~$j$ such that $q'_j\in\bigcup_{j \in
    [h]\setminus B}N(v_j)$. Then~$Q'_i$ is split into the paths
  $q'_1,\dots,q'_{j_1}$ and $q'_{j_1},\dots,q'_{j_2}$ and so on, up to
  $q'_{j_{\ell}},\dots,q'_{s'}$.  By performing this splitting for all~$Q'_i$ we
  obtain, in total, a collection $Q_1, \dots ,Q_\ell$ of pairwise internally
  vertex-disjoint paths which has the following properties for each
  $i=1,\dots,\ell$ by definition of~$A$ and Claim~\ref{cl:main:neighbourB}. All
  internal vertices of $Q_i$ are contained in $X \setminus X_0$, and there is $j
  \in [h]\setminus B$ such that the first vertex of $Q_i$ and the last vertex of
  $Q_{i-1}$ (which might be the same) are both contained in $N(v_j)$, where
  $Q_0=Q_\ell$.
  Again, we allow the degenerate case that we only have $Q_1=C$ with all
  vertices internal.

  Next, for each fixed $i \in [\ell]$, we construct a walk $R_i$ in $H$
  corresponding to the path $Q_i=q_1,\dots,q_s$ whose weight has the same parity
  as the length $s-1$ of $Q_i$. To this end, in the non-degenerate case, let
  $r_1,r_s \in [h]\setminus B$ be such that $q_1 \in N(v_{r_1})$ and $q_s \in
  N(v_{r_s})$. In the degenerate case, we choose~$r_1$ and~$r_s$ later.  Our
  walk~$R_i$ has end-vertices $r_1$ and $r_s$. Recall that
  $q_2,\dots,q_{s-1} \in X \setminus X_0$. We distinguish four cases.
  
  \emph{Non-degenerate case $s=2$:} In this case $r_1\neq r_2$ by~\ref{main:P2} of Claim \ref{cl:main:neighbour}.
  
  as otherwise
  $q_1q_2v_{r_1}$ would form a triangle in $G$. In this case, for $R_i$ we take
  the edge $e=r_1r_2$, which has weight $\omega(e)=3$ because $q_1q_2$ is
  an edge between $N(v_{r_1})$ and $N(v_{r_2})$.

  \emph{Non-degenerate case $s=3$:} If $r_1=r_3$, we can simply take the one
  vertex path $R_i=r_1$.  Otherwise, if $r_1\neq r_3$, let $r_2 \in [h]\setminus
  B$ be such that $q_2$ has a neighbour $x$ in $N(v_{r_2})$, which exists by
  Claim~\ref{cl:main:neighbourB}.  By~\ref{main:P4} of Claim \ref{cl:main:neighbour} if $r_2\neq r_1$ then the edge
  $r_1r_2$ has weight $4$.
  In particular,
  we can choose $r_2=r_3$ with $r_2\neq r_1$ and for $R_i$ we take the edge $r_1r_2=r_1r_3$ with weight $4$.

  \emph{Non-degenerate case $s>3$:} For $j=3,\ldots, s-2$, we use
  Claim~\ref{cl:main:neighbourB} to conclude there is $r_j \in [h]\setminus B$
  such that $q_j$ has a neighbour $y_j$ in $N(v_{r_j})$. We set $r_2=r_1$,
  $r_{s-1}=r_s$, and let $y_2=q_1$, $y_{s-1}=q_s$. Note that with this $q_j$ has
  a neighbour $y_j$ in $N(v_{r_j})$ also for $j=2$ and $j=s-1$. Finally, we define
  $R_i$ as $r_2,r_3,\ldots r_{s-1}$.
  
  We now show that $R_i$ is a walk from $r_1$ to $r_s$ whose weight has the same
  parity as the length of~$Q_i$ also in this case.  First, we observe that $R_i$
  starts at $r_2=r_1$ and ends at $r_{s-1}=r_s$. Next we note that $r_j\neq
  r_{j+1}$ for $j=2,\ldots, s-2$ by \ref{main:P3} of Claim~\ref{cl:main:neighbour}. 
  Finally, \ref{main:P5} of Claim~\ref{cl:main:neighbour} implies that $r_jr_{j+1}$ has weight $3$ or $5$, since by construction, there is a path of length
  $3$ between $N(v_{r_j})$ and $N(v_{r_{j+1}})$ (namely
  $y_j,q_j,q_{j+1},y_{j+1}$).
  Since the weight of each edge in $R_i$ is odd, the weight of $R_i$ has the
  same parity as $s-3$ (the number of edges of $R_i$). Since $Q_i$ has length
  $s-1$, the weight of $R_i$ and the length of~$Q_i$ have the same parity as desired.

  \emph{Degenerate case:} In this case $\ell=1$ and $Q_1$ is a cycle
  $q_1,q_2,\ldots, q_s, q_{s+1}=q_1$ of odd length $s$. For $j=1,\ldots, s$, we
  let $r_j \in [h]\setminus B$ be such that $q_j$ has a neighbour in
  $N(v_{r_j})$. As in the previous case, we conclude that $r_1,\ldots, r_s, r_1$
  is a walk with edges of weight~$3$ or~$5$, hence a closed odd walk.

  \medskip

  This completes the construction of the walks $R_i$ in $H$.  As $C$ was an odd
  cycle in $G$, the sum of the lengths of the $Q_i$ is odd. Further, by
  construction, either we are in the degenerate case when we get one closed odd
  walk, or we are in the non-degenerate case and each walk~$R_i$ ends in the
  same vertex as~$R_{i+1}$ starts in (where indices are taken modulo~$\ell$).  In
  either case, the union of the walks $R_i$ thus is a closed walk of odd weight
  in $H[[h] \setminus B]$ which contains a cycle of odd weight.  This is the
  desired contradiction and, therefore, $G[V\setminus A]$ is bipartite.
\end{claimproof}

Our final claim shows that~\ref{main:third} holds.

\begin{claim}\label{claim5.3}
 $G[A \cup \{ v\}]$ is bipartite for every $v \in V\setminus A$.
\end{claim}
\begin{claimproof}
  Let us assume that, for some $v \in V\setminus A$, there is an odd cycle $C$
  in $G[A \cup \{ v\}]$. There are three cases: either $v=v_{w}$ with $w \in
  [h]\setminus B$, or $v \in N(v_{w})$ with $w \in [h]\setminus B$, or $v \in X
  \setminus X_0$. We start by ruling out the first. Indeed, if $v = v_{w}$ with $w
  \in [h]\setminus B$, then~$v$ cannot be contained in $C$ because $N(v_w)
  \subseteq V\setminus A$, hence $C\subseteq A$. We conclude that in this case
  we can simply choose some new $v\in N(v_w)$ and continue the following
  argument with this~$v$.

  In the other two cases, we proceed as follows. If $v \in X \setminus X_0$, by
  Claim~\ref{cl:main:neighbourB} we can fix a $w \in [h]\setminus B$ such that
  $v$ has a neighbour in $N(v_w)$. Otherwise, we fix $w \in [h]\setminus B$ such that 
  $v \in N(v_{w})$. By assumption $H[B \cup \{ w \}]$ is weighted bipartite.

  Recall that~$A$ consists of $\{v_i\}\cup N(v_i)$ with $i \in B$ and the
  vertices in $X_0$, and that every vertex in~$X_0$ has a neighbour in some
  $N(v_i)$ with $i\in B$.  We want to construct a cycle of odd weight in $H[B
    \cup \{ w \}]$ to obtain a contradiction.  We proceed almost exactly as in
  Claim \ref{claim5.2} and we shall not repeat the details here, but only
  indicate the differences: First of all, the relevant indices are now chosen
  from $B\cup\{w\}$ instead of $[h] \setminus B$, and the internal vertices of
  the paths $Q_1,\dots,Q_t$ come from $X_0$ instead of $X \setminus
  X_0$. Moreover, if $v \in N(v_w)$ and~$v$ appears as an end-vertex of a path
  $Q_i$ then we need to take $w$ for the corresponding end-vertex of the path
  $R_i$.  Similarly, in the case when $v \in X \setminus X_0$ and $v$ appears as
  an internal vertex of a path $Q_i$, we take $w$ as the corresponding vertex in
  the path $R_i$. The remaining arguments work as before, also in these two
  cases.
\end{claimproof}

This completes the proof of Theorem~\ref{thm:main}.
\end{proof}

For the general upper bound $\delta_\chi(\cC_{2k-1},c) \le \tfrac{1}{2k\lfloor c/3 \rfloor}$ we let $G$ be a $\cC_{2k-1}$-free graph of minimum degree at least $\tfrac{1}{2k\lfloor c/3 \rfloor} |V(G)|$ and obtain an auxiliary graph $H \in \cH(k,2k\lfloor c/3 \rfloor)$ in the same way.
Then we can partition $H$ into $\lfloor c/3 \rfloor$ parts of size at most $2k$ and apply Lemma~\ref{LemmaMainWeight345_actual} to each of them.
Almost exactly as above we can then translate the partition of each part back to a $3$-colouring of the corresponding part of $G$, while also taking care of the left-over vertices in $X$, to obtain a $3 \lfloor c/3 \rfloor$-colouring of $G$.


\section{Finding and combining weighted bipartite sets}\label{GoodSection}
In this section, we focus on finding sufficient conditions for a set to be weighted bipartite. We start with the following
lemma, which states that certain balls around a vertex are weighted bipartite.

\begin{lemma}\label{SectionsAreGood}
Let $k\ge 5$ be an integer and~$H$ be a weighted graph with edge weight $\omega\colon E(H)\to\{3,4,5\}$. If $H$ contains no cycle of odd weight smaller than $2k+1$, then for any $u\in V(H)$ we have that $N^{k-3}_\omega[u]$ is weighted bipartite.
\end{lemma}

\begin{proof}[Proof of Lemma \ref{SectionsAreGood}]
  For this proof, it is practical to return to the unweighted setting. Hence, let
  $G$ be the (unweighted) graph obtained from $H$ by replacing every edge of
  weight $s$ by a path with $s$ edges. By construction, all vertices of~$H$ are
  also vertices of~$G$.  Note further that any odd cycle $C$ in $G$ corresponds
  to a cycle in $H$ whose weight is exactly the length of $C$ and vice versa.

  Let us now assume for contradiction that for some $u \in V(H)$ there
  exists a cycle~$C_H$ of odd weight in $N^{k-3}_\omega[u]$, and denote by
  $C$ the corresponding odd cycle in $G$. We define for all non-negative integers $j$, the level sets
  $L_j=\{ x \in V(G) : d_G(u,x)=j \}\subseteq V(G)$ to be the sets containing all vertices in $G$ at distance exactly
  $j$ from $u$, and the set $B=\bigcup_{j=0}^{k-1}L_j$.  We claim that
  $C\subseteq B$. Indeed, for $x \in V(C)\cap V(H)\subseteq V(G)$ we
  have $d_G(u,x) = d_{\omega,H}(u,x)\le k-3$ and, thus, for any $y \in V(C)$ there exists $x \in V(C)\cap V(H)$ with $d_G(u,y) \le d_G(u,x)+2 \le k-1$.

  Since~$C$ is an odd cycle, there must be an edge $xy$ of $C$ with $x$ and $y$
  in the same level set $L_j$. Indeed, otherwise we could properly $2$-colour
  the vertices of the odd cycle $C$ by parity of the level of each vertex.
  We conclude
  that there are a $u,x$-path and a $u,y$-path each with exactly $j\le k-1$ edges.  The
  odd closed walk obtained from these two paths and the edge $xy$ contains an
  odd cycle of length at most $2j+1 \le 2k-1$. But this corresponds to a cycle in
  $H$ of weight odd and smaller than $2k+1$, which contradicts our assumption.
\end{proof}

Lemma~\ref{SectionsAreGood} gives us a large family of sets that are weighted bipartite. This gives us access to many possible candidates for our set $B$. The additional advantage of Lemma~\ref{SectionsAreGood} is that the sets it refers to are very simple, and this makes it easier to interpret our constructions later on.
Our next lemma provides a similarly useful construction, allowing us to build weighted bipartite sets starting from a minimal weight path.

\begin{lemma}
\label{lem:ShortGood}
Let $i \ge 1$ be an integer and $k \ge 10i + 15$.
Let $H$ be a weighted graph with edge weight $\omega\colon E(H)\to\{3,4,5\}$ which contains no cycle of weight odd and smaller than $2k+1$. If $P$ is a path of minimal weight between its end-vertices, then $N^{i}[P]$ is weighted bipartite.
\end{lemma}
\begin{proof}




Assume that there exists a path $P$, which is of minimal weight between its end-vertices and such that $N^{i}[P]$ is not weighted bipartite.
Further, assume that $P$ is minimal with this property, i.e.~for any path $P'$ obtained from $P$ after removing one of its end-vertices, we have that $N^{i}[P']$ is bipartite. Let $P'$ be one of these shortened paths and let $z$ be the end-vertex removed from $P$ to obtain $P'$.
We label the vertices in $N^{i}[P] \setminus N^{i}[P']$ by $w_1,\dots,w_m$ and we take $h$ the minimal index such that $L_h = N^{i}[P'] \cup \{ w_1,\dots,w_h\}$ is not bipartite.
This implies that in $L_h$ there exists a cycle of odd weight. Let $Q$ be one of these cycles, taken of minimal weight. Note that $Q$ has to pass through $w_h$, so we denote with $x$ and $y$ the two neighbours of $w_h$ in $Q$. Let $x'$ and $y'$ be the vertices in $P'$ closest to $x$ and $y$ respectively.

Note that $d_{L_{h-1}}(x',x)\le i+1$ and $d_{L_h}(x,z)\le d_{L_h}(w_h,z)+1\le i+1$. Where the second inequality comes from the fact that by definition $w_h\in N^{i}[P]\setminus N^{i+1}[P\setminus \llb z\rrb]$.
As $P$ is a path of minimal weight between its end-vertices, the same holds for the sub-path between $x'$ and $z$.
Therefore, 
\[d_{\omega,P}(x',z) \le d_{\omega,L_{h-1}}(x',x) + d_{\omega,L_h}(x,z) \le 5 (2i+2)\]
and the analogous argument gives $d_{\omega,P}(y',z) \le 5 (2i+2)$.
This gives $d_{\omega,P}(x', y') \le 5 (2i+2)$ because $x', y'$ and $z$ are in the same path and $z$ is one of the two end-vertices. This also implies $d_{\omega,L_{h-1}}(x,y) \le 10 (2i+2)$.
We let $Q' \subseteq Q$ be the path in $L_{h-1}$ with end-vertices $x$ and $y$.
The parity of $\omega(Q')$ and $d_{\omega,L_{h-1}}(x,y)$ has to be the same, as otherwise there would be a cycle of odd weight in $L_{h-1}$.
But, as $\omega(xw_h)+\omega(yw_h)$ and $\omega(Q')$ have different parity, the parity of $d_{\omega,L_{h-1}}(x,y)$ is also different from the parity of $\omega(xw_h)+\omega(yw_h)$. Therefore, using that $Q$ is the lightest cycle of odd weight, we get that $\omega(Q) \le d_{\omega,L_{h-1}}(x,y) + 10 \le 10 (2i+2)+10$.
This is less than $2k+1$ for our choice of $k$ and gives us the desired contradiction.
\end{proof}

Now that we proved that the most basic sets (paths and balls) have our desired property, we are ready to start the construction of more complicated sets. In particular, the next Lemma shows how to combine two weighted bipartite sets. We need to point out that this combination is not always possible. It might be better to interpret the next result as a condition under which the property of being weighted bipartite is preserved under the union operation.

\begin{lemma}\label{LemmaForGood}
  Let $i\ge 1$ be an integer and let~$H$ be a weighted graph. 
  Let $B_1, B_2$ and $P$ be three sets of vertices in~$H$ such that $d(B_1,B_2) \ge 2i+2$ and $P$ is connected.
  If both $N^{i}[B_1\cup P]$ and $N^{i}[B_2\cup P]$ are weighted bipartite, then
  $N^{i}[B_1\cup B_2\cup P]$ is weighted bipartite.
\end{lemma}

\begin{proof}

  Let $K=B_1\cup B_2\cup P$.
  We want to show that $N^{i}[B_1\cup B_2\cup P]$ is weighted bipartite and let us assume for contradiction that it contains a cycle $C$ of odd weight.
  Let us denote by $B_1'$ the set $N^{i}[B_1]\setminus N^{i}[P]$ and by $B_2'$ the set $N^{i}[B_2]\setminus N^{i}[P]$.
  Since both $N^{i}[B_1\cup P]$ and $N^{i}[B_2\cup P]$ are weighted bipartite, $C$ must intersect both $B_1'$ and $B_2'$.
  Let $y(C)$ be the number of connected components of $C$ induced by $C\cap (B_1'\cup B_2')$ in $H$.
  In other words, $y(C)$ is the number of times that $C$ leaves $B_1'$ or $B_2'$.
  It is possible that $C$ leaves $B_1'$, continues in $N^{i}[P]$, but then returns to $B_1'$ (or the same with $B_2'$), so $y(C)$ does not need to be even, but it has to be at least $2$.
  Assume that $C$ is such that $y(C)$ is minimal.

  Let $w$ be any vertex in $C\cap B_1'$.
  Let $q_1$ and $q_2$ be the end-vertices of the maximal path in $C \cap
  N^{i+1}[B_1]$ containing $w$. That is, $q_1$ and $q_2$ are obtained by moving
  from~$w$ in both possible directions along~$C$ and then taking the first
  vertices that are outside $N^{i+1}[B_1]$.
  Because $q_1,q_2\in N^{i+1}[B_1]\setminus N^{i}[B_1]$, $d(B_1,B_2) \ge 2i+2$, and $C \cap B_2'\neq\emptyset$ it follows that $q_1 \neq q_2$ and $q_1, q_2\in N^{i}[P]$.
  
  As $P$ is connected, there is a path in $N^{i}[P]$ between
  $q_1$ and $q_2$.  Since this path is different from the two paths between
  $q_1$ and $q_2$ in $C$ (as it cannot overlap with $B_1'$ and $B_2'$), we obtain from $C$ at least two cycles in
  $N^{i}[K]$, and at least one of them, let us call it $C'$, has odd weight.
  We have that $y(C')<y(C)$ since we substituted a path in $C$
  containing at least one component of $C\cap
  B_1'$
  (and thus contributing at least
  one to $y(C)$) with a path in $N^{i}[P]$.
  This is a contradiction to the choice of $C$, which was picked with minimal value of $y(C)$.
\end{proof}

We end this section with a corollary of Lemmas~\ref{lem:ShortGood} and ~\ref{LemmaForGood}, which combines the two results in a form that is easier to apply.

\begin{corollary}
\label{Cor:ForGood}
 Let $i\ge 1$ be an integer and $k \ge 10i + 35$.
 Let $H$ be a weighted graph with edge weight $\omega\colon E(H)\to\{3,4,5\}$ which contains no cycle of odd weight that is smaller than $2k+1$.
  Let $B_1, B_2$ be sets of vertices in~$H$ such that $d(B_1,B_2) \ge 2i+2$ and let~$P$ be a path of minimal weight between $B_1$ and $B_2$.
  If $N^{3i+1}[B_1]$ and $N^{3i+1}[B_2]$ are weighted bipartite, then $N^{i}[B_1\cup B_2\cup P]$ is weighted bipartite.
\end{corollary}

\begin{proof}
Let us first establish that $N^i[B_j \cup P]$ is weighted bipartite for $j=1,2$.
Let $P_j$ be the path on the first $2i+2$ vertices of $P$ starting from $B_j$.
Note that~$P_j$ is disjoint from~$B_{3-j}$ since $d(B_1,B_2) \ge 2i+2$.
Let~$P'_j$ be the path, starting in the last vertex of~$P_j$ and ending in $B_{3-j}$.
Then~$P'_j$ is non-empty and since~$P_j$ contains $2i+2$ vertices we have $d(B_j,P_j') \ge 2i+2$.

Next we note that $N^{i}[B_j \cup P_j] \subseteq N^{3i+1}[B_j]$ is weighted bipartite.
Also $N^{i}[P_j \cup P'_j]=N^{i}[P]$ is weighted bipartite by Lemma~\ref{lem:ShortGood} as $P$ is a shortest path between its end-vertices and $k \ge 10i + 15$.
We can then apply Lemma~\ref{LemmaForGood} with $B_j \cup P_j$, $P_j \cup P_j'$, and $P_j$ to deduce that $N^{i}[B_j \cup P_j \cup P'_j] = N^{i}[B_j \cup P]$ is weighted bipartite.
Another application of Lemma~\ref{LemmaForGood} with $B_1$, $B_2$, and $P$ immediately gives that $N^{i}[B_1 \cup B_2 \cup P]$ is weighted bipartite.
\end{proof}

\section{Cycles of odd weight have large neighbourhoods}\label{BipartiteComplement}

We dedicate this section to proving Lemma~\ref{lemlargeweight}, restated here.

\lemlargeweight*

Note that if all edges of the cycle are of weight $3$, the cycle itself has at least $\tfrac{2k+1}{3}$ vertices, but if edges have other weights it might have fever vertices. To overcome this, we use the spanning tree of edges of weight $3$. Indeed, each maximal path of weight $3$ edges in the cycle has a neighbour outside of the cycle. Carefully analysing this situation gives the desired bound.

Before working with cycles, we prove an analogous result for paths, which we use to prove the former.
Before proceeding, we also remark that lemmas in this section are not stated in terms of $\mathcal{H}(k)$ as we want to apply them in more generality.

\begin{lemma}\label{lem:pathneighbour}
Let $F$ be a weighted graph with edge weight $\omega\colon E(F)\to\{3,4,5\}$. Assume that $F=T\cup P$, where $T$ is a spanning tree in which all edges have weight~$3$ and $P$ is a non-spanning path of weight $\ell$ with end-vertices $x$ and $y$. If $F$ has no
cycles of weight~$11$ and~$P$ has minimal weight among all $x,y$-paths in $N^1_F[P]$,
then $\vass{N^1_F[P]}\geq \frac{\ell}{3}+\frac{5}{3}$.
\end{lemma}

\begin{proof}
We write $P=Q_1,\dots,Q_s$ as a concatenation of (possibly trivial) sub-paths~$Q_i$ such that within each~$Q_i$ all edges have weight~$3$ and the edge $e_i$ between $Q_i$ and $Q_{i+1}$ has weight $\omega(e_i)>3$. 
 
If $s=1$ then each edge of $P$ has weight $3$ and we are done because $\vass{N^1[P]}>|P|=\frac{\ell}{3}+1$, where the strict inequality comes from the fact that $F$ is connected, and hence $P$ has a neighbour in  $V(F)\setminus V(P)$ (which is not empty because $P$ is not spanning).
 
Assume now that $s$ is at least $2$. Since~$T$ is a spanning tree in~$F$, for each $Q_i$ we can fix a vertex $z_i\in N^1[Q_i]\setminus P$ and a vertex ~$x_i$ in $Q_i$ such that $z_ix_i$ has weight $3$.

For $i<j$ we have $z_i\neq z_j$ unless $j=i+1$ and $e_i=x_ix_{i+1}$, because
otherwise $P$ would not be an $x,y$-path of minimal weight in its neighbourhood. If
$z_i=z_{i+1}$, we say that $(i,i+1)$ has a \emph{hat}. In this case we also know
that $e_i=x_ix_{i+1}$ has weight $4$, as otherwise $x_i,x_{i+1},z_i$ would form a
cycle of weight~$11$. Moreover, neither $(i-1,i)$ nor $(i+1,i+2)$ has a hat
(otherwise we would have without loss of generality that $z_{i-1}=z_{i+1}$ and
we could replace the sub-path $x_{i-1}, x_{i}, x_{i+1}$ of weight eight with
$x_{i-1}, z_{i+1}, x_{i+1}$ of weight six).
 
 Now, if $(i,i+1)$ has a hat, we ``merge'' $Q_i$ and $Q_{i+1}$:
 We rewrite $P=Q'_1,\dots,Q'_{s'}$ such that each $Q'_j$ either is the concatenation $Q_i,Q_{i+1}$ for some $i$ such that $(i,i+1)$ has a hat, or is $Q_i$ for some $i$ such that neither $(i-1,i)$ nor $(i,i+1)$ has a hat. 
 In the former case, we say that $Q'_j$ was formed by a hat.
 In both cases, we set $z'_j=z_i$. Observe that by construction $z'_j\neq z'_{j'}$ for $j\neq j'$.
 We thus conclude that we have 
 \[|N^1[P]|\ge\sum_{j\in[s']}(|Q'_j|+1)=s'+|P|\,.\]
 Moreover, since $Q'_j$ and $Q'_{j+1}$ are connected by an edge of weight at most~$5$, we have
 \[\omega(P)\le 5(s'-1)+\sum_{j\in[s']}\omega(Q'_j)\,.\]
 If $Q'_j$ was formed from a hat, then $\omega(Q'_j)=3(|Q'_j|-2)+4=3|Q'_j|-2$
 and otherwise $\omega(Q'_j)=3(|Q'_j|-1)\le 3|Q'_j|-2$.
 Therefore,
 \begin{equation*}
    \omega(P)\le 5(s'-1)+\sum_{j\in[s']} (3|Q'_j|-2)=3s'-5+3|P|\,,
 \end{equation*}
 and hence $|N^1[P]|\ge \vass{P}+s'\ge \frac{\omega(P)}{3}+\frac{5}{3}$ as desired.
\end{proof}

We are now ready to present our proof of Lemma~\ref{lemlargeweight}, which provides a similar lower bound on the size of the neighbourhood of a non-spanning cycle of odd weight in a graph $H \in \mathcal{H}(k)$.

\begin{figure}
\centering
\begin{tikzpicture}
\node [circle, draw, fill=white, inner sep=2pt, minimum width=1pt] (yi) at (-1,0) {$y_i$};
\node [circle, draw, fill=white, inner sep=2pt, minimum width=1pt] (xi) at (0,0) {$x_i$};
\node [circle, draw, fill=white, inner sep=2pt, minimum width=1pt] (w) at (1,0) {$w$};
\node [circle, draw, fill=white, inner sep=2pt, minimum width=1pt] (xj) at (2,0) {$x_j$};
\node [circle, draw, fill=white, inner sep=2pt, minimum width=1pt] (z) at (1,1) {$z_i$};
\node [circle, draw, fill=white, inner sep=2pt, minimum width=1pt] (yj) at (3,0) {$y_j$};
\draw[thick] (yi) -- (xi) -- (w) -- (xj) -- (z) -- (xi);
\draw[thick] (xj) -- (yj);
\end{tikzpicture}
\caption{Shortcut in the case $d(C)=2$, where $x_iz_i$ and $x_jz_i$ are edges of weight $3$ and the sum of the weights of $x_iw$ and $x_jw$ is at least $7$.}\label{fig:shortcut}
\end{figure}
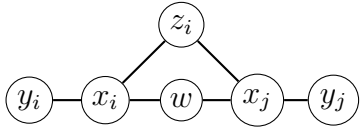
\begin{proof}[Proof of Lemma \ref{lemlargeweight}]
Let $T$ be a spanning tree of edges of weight $3$ associated to $H$.
Let $C$ be a cycle of odd weight (note that the weight has to be at least $2k+1$ because $H\in \cH(k)$).
We write $C=Q_1,\dots,Q_s$ as a concatenation of
sub-paths~$Q_i$ such that within each~$Q_i$ all edges have weight~$3$ and the
edge $e_i$ between $Q_i$ and $Q_{i+1}$ has weight $\omega(e_i)>3$. We also call these
sub-paths~$Q_i$ the \emph{segments} of the cycle.
 
Since~$T$ is a spanning tree in~$H$, for each segment $Q_i$ we can fix a vertex $x_i$ in
$Q_i$ and a vertex $z_i\in N^1[Q_i]\setminus C$ such that $x_iz_i$ is in $T$ (and therefore has weight $3$).
Let $I(C) =C \cup \{ z_i : i \in [s] \}$ and note that it suffices to show that
$|I(C)| \geq \frac{2k+1}{3}$. We also point out that
$I(C)$ might depend on the choice of the $x_i$ and $z_i$. Since by removing edges in $H$ we cannot increase
the size of $N^1[C]$, it suffices to consider the graph $H(C)$ with vertex set
$I(C)$ and edge set $E(C)\cup \llb x_iz_i: i\in [s]\rrb$.    We now could be temped
to immediately apply Lemma \ref{lem:pathneighbour} to some spanning path~$P$
in~$C$.  However, this is not possible since~$P$ may not have minimal weight in
$N^1_{H(C)}[P]$.  Therefore, our goal is to ``move'' to a (possibly) different
cycle $C'$ in which we do not encounter this issue.

\begin{claim}
  There is a non-spanning cycle~$C'$ of odd weight in~$H$ such that~$C'$ has minimal weight
  among all cycles of odd weight in~$H(C')$ (defined analogously as above) and such that $|I(C)| \ge |I(C')|$.
\end{claim}
\begin{claimproof}  
  We shall move through a sequence $C_1,C_2,\dots$ of cycles until we obtain a
  cycle $C_{\tilde\ell}=C'$ with the desired properties, where from one cycle~$C_\ell$ to
  the next~$C_{\ell+1}$ we do not increase the weight, we decrease the number~$s$
  of segments, and we have $|I(C_\ell)|\ge|I(C_{\ell+1})|$.  If
  successful, this process terminates since we always decrease the value of $s$ and if $s=1$ for some cycle $C_\ell$,
  then~$C_\ell$ is the only cycle in $I(C_\ell)$. We set $C_1=C$ and now assume that
  we currently have a cycle~$C_\ell$ with~$s$ segments, for which we either want to show it has minimal weight among all cycles of odd weight in $H(C_\ell)$
  or move to a new cycle~$C_{\ell+1}$ with the properties just specified.
  
  Let $d(C_\ell)$ be the maximum (unweighted) distance on $C_\ell$ between $x_i$ and $x_j$
  such that $z_i=z_j$ over all choices $i,j \in [s]$.
  If $d(C_\ell) \le 1$ then $C_\ell$
  is of minimal weight in $H(C_\ell)$ (the path $x_iz_ix_j$ has larger weight than the edge $x_ix_j$) and we are done.

  Next, assume $d(C_\ell)=2$.  
  Let $x_i, x_j$ be two vertices of distance $2$ on $C_\ell$ such that
  $z_i=z_j$.  Let $w$ be the vertex in $C_\ell$ adjacent to both $x_i$ and $x_j$ and
  let $y_i$ and $y_j$ be the other neighbours on $C_\ell$ of $x_i$ and $x_j$,
  respectively, as shown in Figure~\ref{fig:shortcut}.
  The cycle $x_i z_i x_j w$ is even, because it has weight at most $16$ and $k\ge 8$.
  Therefore, the cycle $C_{\ell+1}$
  obtained from $C_\ell$ by replacing $w$ with $z_i$ is of odd
  weight. Since~$x_i$ and~$x_j$ are in different segments and $x_iz_i,x_jz_j$
  have weight~$3$, it also holds that $C_{\ell+1}$ has smaller weight
  than~$C_\ell$.
  Moreover, $C_{\ell+1}$ partitions into the same segments
  $Q_1,\dots,Q_s$ except that $Q_i$ and $Q_j$ are replaced by $Q_i'=(Q_i \cup
  Q_j \cup \{ z_i\}) \setminus \{ w\}$ and potentially a singleton segment $\{w\}$
  is removed. In particular, the number of segments decreases.
  Similarly, $x_1,\dots,x_s$ and $z_1,\dots,z_s$ can be chosen the same for~$C_{\ell+1}$
  except that $x_j$ and $z_j$ are removed and $z_i$ is replaced by a neighbour of
  $Q_i'$ in the spanning tree $T$ that is outside of $C_{\ell+1}$ (and might be $w$).
  Note that we get $|I(C_{\ell+1})|\leq|I(C_\ell)|$.
  
  Finally, assume that $d(C_\ell) \ge 3$. Let $x_i, x_j$ be two vertices of
  distance $d(C_\ell)$ on $C_\ell$ such that $z_i=z_j$. Then $x_i,z_i,x_j$ is
  a ``shortcut'', that is, a path that is of smaller weight than both
  $x_i,x_j$-paths $P_\circlearrowleft$ and $P_\circlearrowright$ on~$C_\ell$. Hence, either $x_i,z_i,x_j$ together with
  $P_\circlearrowleft$ or $x_i,z_i,x_j$ together with
  $P_\circlearrowright$ gives a cycle~$C_{\ell+1}$ of odd weight that is smaller than that of~$C_\ell$.
  Since $\vass{C_{\ell+1}\setminus C_\ell}=1$, the new cycle~$C_{\ell+1}$
  partitions into some segments of~$C_\ell$ and an additional segment $Q'$
  containing $x_i,z_i,x_j$ and possibly some more vertices of the two segments
  of~$C_\ell$ containing $x_i$ and $x_j$. The number of segments in this
  partition of~$C_{\ell+1}$ is less than~$s$ (since $x_i$ and $x_j$ were in different segments by definition).  Moreover, for the segments $Q_q$
  retained from~$C_\ell$ we keep the vertices $x_{q}$ and $z_{q}$ as before,
  and for the new segment $Q'$ we pick vertices~$x'$ and $z'$ such that $x'$
  is in~$Q'$ and~$z'$ is a neighbour of~$x'$ in~$T$ outside $C_{\ell+1}$, where
  potentially $z'$ is not contained in $I(C_\ell)$. With these choices we get
  \begin{align*}
    \vass{I(C_{\ell+1})}
    &=\vass{C_{\ell+1}}+\vass{I(C_{\ell+1})\setminus C_{\ell+1}}\leq \vass{C_\ell}-1+\vass{I(C_\ell)\setminus C_\ell}\\
    &<\vass{I(C_\ell)\setminus C_\ell}+\vass{C_\ell}=\vass{I(C_\ell)}\,,
  \end{align*}
  as required and hence also successfully constructed~$C_{\ell+1}$ in this case.
\end{claimproof}

Let~$C'$ be an odd cycle such as the one promised by this claim, let~$s'$ be the number of its segments, and let $z_1',\dots,z'_{s'}$ be the neighbours of the segments.
Our goal now is to argue that $|I(C')|\ge\frac{2k+1}3$, which proves the lemma since $\vass{I(C')}\leq \vass{I(C)}$.
If $s'=1$, then all but at most
one edge of~$C'$ have weight $3$ and thus $|I(C')| \ge |C'|+1 \ge \tfrac{2k+1-5}{3}
+2 \ge \tfrac{2k+1}{3}$. For the first inequality, we used that there is a vertex in $I(C')\setminus
C'$, which is true because $T$ is connected and $C'$ is not spanning in $H$. Hence, assume from now on that $s'\ge 2$.

We can fix one edge $e$ in $C'$ that is not in the spanning tree $T$.  Removing~$e$ from $C'$, we obtain a path $P$ of weight at least $(2k+1)-5$.  We
let $T'$ be the graph consisting of all edges of weight $3$ in $H(C')$ except
$e$ and one additional auxiliary vertex $v$ connected to each of $z'_1,\dots,z'_{s'}$
with an edge of weight $3$. Observe that~$T'$ is a tree.  Now, consider the
graph $F=(T' \cup H(C')) \setminus \{e\}$. We have that $V(F) \setminus V(P)\not=
\emptyset$ (since $v\in V(F) \setminus V(P)$), and~$P$ has minimal weight among all paths in $N^1_F[P]$ connecting
its end-vertices since $v$ is not contained in~$N^1_F[P]$ and by
the minimality of~$C'$. Moreover, $F$ has no cycles of weight~$11$ since any
such cycle would need to include the auxiliary vertex which is only connected to
the vertices $z'_1,\dots,z'_{s'}$ which form an independent set in~$F$, and thus
any cycle using the auxiliary vertex has weight at least~$12$.  We conclude that
we can apply Lemma~\ref{lem:pathneighbour} to $P$ and $T'$ to get
$I(C') \ge N^1_F[P]\ge
\tfrac{2k-4}{3}+\tfrac{5}{3} \ge \tfrac{2k+1}{3}$ as required.
\end{proof}

We end this section with a useful corollary of Lemma~\ref{lemlargeweight}.
\begin{corollary}\label{lem:pathneighbour2}
Let $\ell \ge 13$ be an odd integer.
Let $F$ be a weighted graph with edge weight $\omega\colon E(F)\to\{3,4,5\}$. Assume that $F=T\cup P$, where $T$ is a spanning tree in which all edges have weight~$3$ and $P$ is a non-spanning path with end-vertices $x$ and $y$.
If $F$ has no cycles of odd weight below~$\ell+4$, and the minimal weight of an $x,y$-paths in $F$ is at least~$\ell$, then $\vass{N^1_F[P]}\geq \frac{\ell}{3}+\frac{4}{3}$.
\end{corollary}

\begin{proof}
We add the edge $xy$ to $F$ and define its weight to be $s \in \{4,5\}$ such that $\omega(P)+s$ is odd.
Let $C$ in $F$ be that cycle of odd weight consisting of $P$ and the edge $xy$.
Because any $x,y$-path is of weight at least $\ell$, in $F$ there is no cycle whose weight is odd and smaller than $\ell+4$.
Therefore, $F \in \mathcal{H}(\tfrac 12 (\ell+4-1))$.
We can then apply Lemma~\ref{lemlargeweight} with $k=\tfrac 12 (\ell+4-1) \ge 8$ an integer (since $\ell$ is odd), to get $|N^1_F[P]|=|N^1_F[C]| \ge \tfrac{\ell+4}{3}$.
\end{proof}

\section{Proof of the main technical lemma}\label{SectionActualProof}

The main objective of this section is to prove our main technical lemma (Lemma~\ref{LemmaMainWeight345_actual}).
We need some further preparations. In the previous sections we first showed how to generate a candidate set $B$ with weight-bipartite neighbourhood, and then how to guarantee that $H\setminus B$ is weighted-bipartite by analysing the size of $\Int(B)$. However, we did not combine yet results of these two types.

Observe that just taking a ball with Lemma~\ref{SectionsAreGood} might only give a small set, while even cleverly removing a few vertices from a cycle of odd weight not necessarily makes it weighted bipartite.
Therefore the first result of this section (and the last piece missing in order to prove Lemma~\ref{LemmaMainWeight345_actual}), is a lemma combining these two. Indeed we show how we can create a candidate set which is weighted bipartite and with a lower bound on its size.

\begin{lemma}\label{lem:PathInCycle}
Let $i\geq 2$ and $k \ge 5i+16$ be integers, and $H \in \mathcal{H}(k)$.
For any $C$ odd cycle in $H$ and $p \in V(C)$ the following holds.
There exists a path $P$ such that $E(P)\subseteq E(C)$, $p \in V(P)$, $N^i[P]$ is weighted bipartite, and $|N^1[P]| \ge \tfrac 23 k - \tfrac{10}{3} i - 5$.

Moreover, we can guarantee that either $p$ is in the unweighted middle of $P$ (the lengths of the paths from the end-vertices of $P$ to $p$ differ by at most $1$) or in the weighted middle of $P$ (the weights of the paths from the end-vertices of $P$ to $p$ differ by at most $5$).
\end{lemma}

\begin{proof}
We construct $P$ as the output of the following recursive procedure starting with $V(P)=\llb p\rrb$.
In each step, we denote with $u$ one of the vertices in $C\setminus P$ adjacent in $C$ to $P$.
If $N^i[P\cup \llb u \rrb]$ is weighted bipartite, we update $P$ to be the path obtained by extending $P$ to the vertex $u$ using the edge in $C$ which connects $u$ to $P$, and then repeat the step. Otherwise we stop the process and output $P$.
Note that this process is well defined, since $N^i[\llb p\rrb]\subseteq N^{5i}_{\omega}[\llb p \rrb]$ and the latter is weighted bipartite by Lemma~\ref{SectionsAreGood}; while $N^i[C]$ is not.
Also note that $N^i[P]$ is weighted bipartite by definition and, therefore, it suffices to prove that $|N^{1}[P]|\geq \tfrac 23 k - \tfrac{10}{3} i - 5$.

If we want to ensure that $p$ is in the unweighted middle of $P$, we simply alternate between extending both ends of $P$ when we select $u$.
For the weighted middle, we choose the $u$ whose weighted distance to $p$ in $C$ is shorter.
In both cases it is easy to see that this guarantees that $p$ is in the desired position.

Once the process stops, let us denote with $x, y$ the two end-vertices of $P$ and with $u_x, u_y$ the two vertices in $C\setminus P$ adjacent in $C$ to $x$ and $y$ respectively (without loss of generality we assume we could not extend $P$ to contain $u_x$).
In order to get the lower bound on $|N^{1}[P]|$, we want to use Corollary \ref{lem:pathneighbour2} with input $\ell=2k-10i-19 \ge 13$. However, it is not immediately clear what graph to use as host graph $F$. Since $H$ has an associated spanning tree $T_H$ with edges of weight $3$, the most immediate choice would be to use the graph obtained by removing from $H$ all the edges not in $P$ or in $T_H$, but there is no guarantee that in this graph there is no short path between $x$ and $y$. 

So we want to build a host graph $F$ with the properties needed to use Corollary \ref{lem:pathneighbour2} with input $\ell$. I.e. we want a graph $F$ which contains (besides $P$) only edges of a spanning tree $T$ composed of edges of weight $3$. Also, we want that in $F$ there are no odd cycles of weight less than $\ell+4$, and that the path of minimal weight in $F$ between $x$ and $y$ has weight at least $\ell$. 

In order to do this, we attach to $F'=N^i[P]\cap (P\cup T_H)$ a forest in such a way that in the resulting graph $F$ the properties are satisfied (no short paths between $x$ and $y$, no short odd cycles, and only edges from a tree or in $P$). We construct a new graph $F$ from $F'$ by adding some new vertices and then edges of weight $3$ in such a way that $F=T\cup P$ (where $T$ is a spanning tree of edges of weight $3$), and such that there are no odd cycles in $F$ of weight less than $\ell+4$ and finally there are no $x,y$-paths of minimal weight that use any edge outside of $N^i[P]\cap (P\cup T_H)$.
This is possible by connecting the components of the forest of  $N^i[P]\cap (T_H\setminus P)$ with long paths of new vertices and new edges of weight $3$ in an acyclic manner.

Even if $F$ is not contained in $H$, its utility comes from the fact that we have $N_F^1[P]\subseteq N_H^1[P]$. It is actually sufficient to show that there is a sub-path of $P$ such that its neighbourhood in $F$ is large enough. For $x', y'$ vertices in $P$ we write $P_{x'y'}$ to denote the sub-path of $P$ between those vertices. With notation, and using Corollary  \ref{lem:pathneighbour2}, it suffices to find any two vertices $x', y'$ in $P$ at weighted distance at least $\ell$ in $F$. By construction, this is equivalent to find $x', y'$ in $P$ at distance at least $\ell$ in $L=N^i[P]\cap (P\cup T_H)$.

We now want to find such $x', y'$. By construction, we do not have any odd cycles in $N^i[P]$ but we do have an odd cycle in $N^i[P\cup\llb u_x\rrb]$ by choice of $u_x$. However, since there might be many vertices in $N^i[\llb u_x \rrb]\setminus N^i[P]$, in order to use the condition that we have no short odd cycles in $H$ we proceed cautiously as follows. Let us fix an arbitrary order of the vertices of $N^i[\llb u_x \rrb]\setminus N^i[P]=\llb z_1,\dots{}, z_m\rrb$ and let $h$ be the minimum index such that $L_h=N^i[P]\cup \llb z_1, \dots{}, z_h\rrb$ contains an odd cycle, and let $Q$ be an odd cycle of minimal weight in $L_h$. We are not interested in the order for its own sake, but because of the ordering and of the definition of $L_h$, we have that $Q$ passes through $z_h$, and that there are no cycles of odd weight in $L_{h-1}$.

Let $x'', y''$ be the two neighbours of $z_h$ in $Q$ and let $x', y'\in P$ be vertices in $P$ closest in $L_{h-1}$ respectively to $x''$ and $y''$.
We claim that $d_{\omega, L_{h-1}}(x'', y'')\geq 2k-9$ (weighted distance in $L_{h-1}$ between $x''$ and $y''$).
Indeed, assume that this is not true and let $P'$ be the shortest path in $L_{h-1}$ between $x''$ and $y''$, and let $P_Q$ be the path in $Q\cap L_{h-1}$ with end-vertices $x''$ and $y''$.
The parity of $\omega(P')$ and $\omega(P_Q)$ can not be different, because then there would be a cycle of odd weight in $L_{h-1}$.
But the parity of $\omega(x''z_h)+\omega(z_hy'') \le 10$ is different from $\omega(P_Q)$ and, therefore, from the parity of $\omega(P')$. This means that $P'$ (which is a path between $x''$ and $y''$) together with $z_h$ gives a cycle of odd weight that is at most $2k$.
This is a contradiction and proves the claimed bound.
Since $d_{\omega, L_{h-1}}(x'', x'), d_{\omega, L_{h-1}}(y'', y')\leq 5 \cdot (i+1)$, we have that there are two points $x', y'\in P$ at distance at least $2k-9-2\cdot 5 \cdot(i+1)=\ell$ in $L_{h-1}$ (and therefore in $L$), as wanted.
\end{proof}

Now we are ready for proceeding to the main proof of this section.
The strategy bringing all this together is to use results of Section~\ref{GoodSection} to combine together constructions such as the one in Lemma~\ref{lem:PathInCycle} whenever easier constructions, like balls, do not work.
We restate and prove the main lemma.
\MainLemma*

\begin{proof}[Proof of Lemma~\ref{LemmaMainWeight345_actual}]
Let us fix the functions $\ell_0=\ell_0(t)= 610 + 5t$ and $k(t) = 9 \ell_0$. 
Let $t$ be a natural number and $k\geq k(t)$. In particular, note that since $k\geq 5490$, we have that $k$ is large enough to apply Lemma~\ref{lem:PathInCycle} with $i=16$ and Corollary~\ref{Cor:ForGood} with $i=3$.
Let $H \in \mathcal{H}(k,2k+t)$ and distinguish two cases.

\textbf{Case A.}
In this case, we assume that in $H$ there are two cycles~$C_1$ and~$C_2$ of odd weight at weighted distance at least $\ell_0$ from each other. Let~$P$ be a path of minimal weight between $C_1$ and $C_2$ (in particular we have $\omega(P)\geq \ell_0$). Let $p_1$ and $p_2$ be the end-vertices of~$P$ in~$C_1$ and~$C_2$, respectively.
For $j=1,2$, let $B_j$ be the path in $C_j$ (with $p_j$ in the unweighted middle of $B_j$) given by Lemma~\ref{lem:PathInCycle} with $i=15$, and note that $d(B_1,B_2) \ge 8$.
See Figure~\ref{fig:CaseA} for an illustration of the situation. We denote by $B$ the set $N^2[B_1\cup B_2\cup P]$.
As $N^{10}[B_j]$ is weighted bipartite by construction for $j=1,2$, we have that $N^1[B]$ is weighted bipartite by Corollary~\ref{Cor:ForGood} applied with $i=3$.

\begin{figure}[htbp]
    \centering
    \begin{tikzpicture}[scale=1.3]
\draw[] ($(1.3,0)+(-180:1.3)$) arc (-180:180:1.3cm);
\draw[] ($(-3.9,0)+(-180:1.3)$) arc (-180:180:1.3cm);
\node[] () at ($(-2.7,1.3)$) {\small{$C_1$}};
\node[] () at ($(0,1.3)$) {\small{$C_2$}};

\node[circle, draw, inner sep=0pt, minimum width=4pt, color=black, fill=black!20, label={[label distance=-0cm]180:\scriptsize{$p_1$}}] (Vu) at ($(-1.3,0)+(-180:1.3)$) {};
\node[circle, draw, inner sep=0pt, minimum width=4pt, color=black, fill=black!20, label={[label distance=-0cm]0:\scriptsize{$p_2$}}] (t) at ($(0,0)$) {};
\draw[] (Vu) -- (t);
\node[label={[label distance=-0.08cm]90:\scriptsize{$P$}}] () at ($(-1.3,0)$) {};

\begin{scope}[on background layer]
\draw[line width=10pt, color=blue!20, line cap=round] ($(1.3,0)+(50:1.3)$) arc (50:310:1.3cm);
\node[] () at ($(0,-1.3)$) {\textcolor{blue}{\small{$B_2$}}};

\draw[line width=10pt, color=blue!20, line cap=round] ($(-3.9,0)+(-165:1.3)$) arc (-165:165:1.3cm);
\node[] () at ($(-2.6,-1.3)$) {\textcolor{blue}{\small{$B_1$}}};
\end{scope}

\end{tikzpicture}
    \caption{Two cycles $C_1$ and $C_2$ at weighted distance at least $\ell_0$ and the construction of the weighted bipartite set.}
    \label{fig:CaseA}
\end{figure}
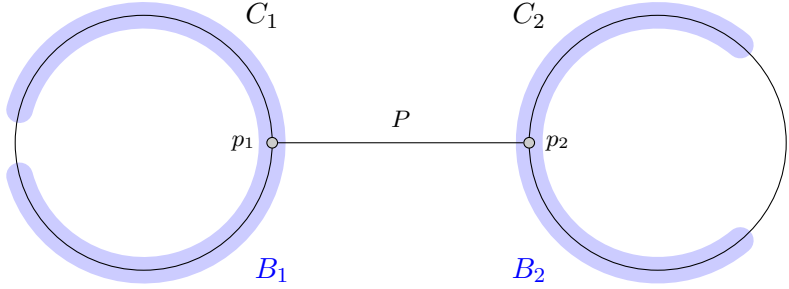

Since we showed that $N^1[B]$ weighted bipartite, it remains to show that $H\setminus B$ is weighted bipartite.
In view of Lemma~\ref{lemlargeInt}, it thus suffices to show that $|\Int(B)|=|N^1[B_1 \cup B_2 \cup P]| \ge \frac 43 k + t$.
As $B_j$ is given by Lemma~\ref{lem:PathInCycle} with $i=16$, we have $|N^1[B_j]| \ge \frac{2k}{3}-59$.
Therefore, as $N^1[B_1]\cap N^1[B_2]=\emptyset$ and $|P\cap N^1[B_j]| \le 2$ for $j=1,2$, we get
\[ |N^1[B_1 \cup B_2 \cup P]| \ge |N^1[B_1]|+|N^1[B_2]|+|P|-4> \tfrac{4}{3} k - 118 + \tfrac 15 \ell_0-4 \ge \tfrac 43 k + t \, ,\]
where the last inequality uses the lower bound $\ell_0=\ell_0(t)\ge 5 \cdot (122+t) = 610 + 5t$.

\smallskip

\textbf{Case B.}
We assume that the first case does not apply and, therefore, between any two odd cycles there is a path of weight at most $\ell_0$.

Let $C_1$ be a cycle of minimal odd weight (note we can assume $\omega(C_1) = 2k+1$ without loss of generality) and let $x_1$ be a vertex in $C_1$.
We let $T_1 = N_\omega^{k-8}[x_1]$ and note that $N[T_1]$ is weighted bipartite by Lemma~\ref{SectionsAreGood}.
Either the complement of $T_1$ induces a weighted bipartite subgraph of $H$, in which case we are done with $B=T_1$, or we can find a cycle $C_2$ of odd weight outside of $T_1$ such that in the complement of $T_1$ there are no cycles of shorter odd weight.

Assume we are in the second case and fix such a $C_2$.
Let $x_2$ be a vertex in $C_1\setminus N[T_1]$ that minimises the distance $d_\omega(x_2,C_2)$. We have $d_\omega(x_2,C_2) \le 2\ell_0$. Indeed, consider a vertex $x_2'$ in $C_1$ such that $d_\omega(x_2',C_2) \le \ell_0$; we have that $d_\omega(x_2',C_1\setminus N[T_1]) \le \ell_0$ since any path of weight at most $\ell_0$ starting outside of $N[T_1]\subseteq N_{\omega}^{k-3}[x_1]$ cannot have an end-vertex at weighted distance less than $k-3-\ell_0$ from $x_1$. Therefore we can find a vertex in $C_1\setminus N[T_1]$ at distance at most $2\ell_0$ from $C_2$.

We let $T_2 = N_\omega^{k-8}[x_2]$ and note that $N[T_2]$ is weighted bipartite by Lemma~\ref{SectionsAreGood}.
For future reference, we note that there are at most two vertices in each of $C_1 \setminus N[T_1]$ and $C_1 \setminus N[T_2]$ and that $k-2  \le d_\omega(x_1,x_2) \le k$, because $x_2 \in C_1 \setminus N[T_1]$ and $\omega(C_1) = 2k+1$.
Either the complement of $T_2$ induces a weighted bipartite subgraph of $H$, in which case we are done with $B=T_2$, or we can find a cycle $C_3$ of odd weight outside of $T_2$ such that in the complement of $T_2$ there are no cycles of shorter odd weight.

Assume we are in the second case and fix such a $C_3$.
In Figure~\ref{fig:CaseB1} we illustrate this situation and the following argument.
By assumption, we have that between $C_1$ and $C_3$, and between $C_2$ and $C_3$, there are paths of weight at most $\ell_0$.
Therefore, there are $s_j \in C_j$ for $j=1,2$ such that $d_\omega(s_1, C_3), d_\omega(s_2, C_3) \le \ell_0$.
Let $P_1, P_2$ be the path in $C_1, C_2$ given by Lemma~\ref{lem:PathInCycle} with $i=16$, and respectively $x_1$ and $s_2$ in the weighted middle. Denote with $p, q$ the end-vertices of $P_1$.
We claim that $d_\omega(x_1,p),d_\omega(x_1,q) \le k-43$.
Indeed, we know from Lemma~\ref{lem:PathInCycle} that $|d_\omega(x_1,p)-d_\omega(x_1,q)| \le 5$ and if say $d_\omega(x_1,p)> k-43$ we would also get  $d_\omega(x_2,q) \ge k-47$. But this implies that $C_1 \subseteq N^{16}[P_1]$. Indeed, continuing along $P_1$ from each of $p,q$ towards $x_2$ for $16$ steps of weight at least $3$ each, we see with $k-47+16 \cdot 3 \ge k+1$ that $C_1 \subseteq N^{16}[P_1]$ since the weight of $C_1$ is $2k+1$. This is in contradiction to the fact that $N^{16}[P_1]$ is weighted bipartite.
This implies $N^{7}[P_1]\subseteq N^{35}_\omega[P_1] \subseteq N^{k-8}_\omega[x_1]=T_1$.
This, together with $T_1 \cap C_2 = \emptyset$, implies that $d(P_1,P_2) \ge 8$. 

Now consider the lightest path $P$ in the whole graph that connects $s_1$ and $s_2$ (respectively the fixed vertices in $C_1$ and $C_2$ closest to $C_3$).
We let $s_1'$ (and $s_2'$) be the last (respectively first) vertex of $P$ on $P_1$ (respectively $P_2$), and denote by $P'$ the sub-path of $P$ between $s_1'$ and $s_2'$.
We now apply Corollary~\ref{Cor:ForGood} with $i=3$ to the sets $B_1=P_1$ and $B_2=P_2$. This gives us that $N^3[P_1 \cup P_2 \cup P']$ is weighted bipartite.

If $P'$ has weight at least $\ell_0$, then $P'$ has at least $\tfrac 15 \ell_0 \ge 122+t$ vertices, and we are done with $B=N^2[P_1\cup P_2\cup P']$ by Lemma~\ref{lemlargeInt}, as
\[ |\Int(B)| \ge |N^1[P_1 \cup P_2 \cup P']| \ge |N^1[P_1]|+|N^1[P_2]|+|P'|-4 \ge 2 (\tfrac23 k - 59) + |P'| - 4 \ge \tfrac43 k + t \, . \]

\begin{figure}[htbp]
    \centering
    \begin{tikzpicture}[scale=1.3]
\draw[] ($(1,0)+(-180:1.3)$) arc (-180:180:1.3cm);
\draw[] ($(-2,0)+(-180:1.3)$) arc (-180:180:1.3cm);
\node[] () at ($(-2,1)$) {\small{$C_1$}};
\node[] () at ($(1,1)$) {\small{$C_2$}};


\foreach \x/\y/\z in {x_2/30/210, s_1'/-40/140, s_1/-90/90, x_1/-120/60}{
\node[circle, draw, inner sep=0pt, minimum width=4pt, color=black, fill=black!20, label={[label distance=-0.1cm]\z:\scriptsize{$\x$}}] (\x) at ($(-2,0)+(\y:1.3)$) {};
}

\foreach \x/\y/\z in {s_2'/-140/40, s_2/-90/90}{
\node[circle, draw, inner sep=0pt, minimum width=4pt, color=black, fill=black!20, label={[label distance=-0.1cm]\z:\scriptsize{$\x$}}] (\x) at ($(1,0)+(\y:1.3)$) {};
}

\draw (-0.5,-2.7) ellipse (3.4cm and 1cm);
\node[] () at ($(-0.5,-2)$) {\small{$C_3$}};

\begin{scope}[on background layer]
\draw[line width=10pt, color=blue!20, line cap=round] ($(-2,0)+(45:1.3)$) arc (45:375:1.3cm);
\node[] () at ($(-2.9,0)$) {\textcolor{blue}{\small{$P_1$}}};

\draw[line width=10pt, color=blue!20, line cap=round] ($(1,0)+(-175:1.3)$) arc (-175:80:1.3cm);
\node[] () at ($(1.8,0)$) {\textcolor{blue}{\small{$P_2$}}};
\draw[] (s_1')--(s_2');
\node[] () at ($(-0.5,-1.1)$) {\small{$P'$}};

\draw[line width=10pt, color=gray!50, line cap=round] (x_2) -- ($(x_2)+(0.7,0)$);

\draw[line width=10pt, color=gray!50, line cap=round] (s_1) -- ($(s_1)+(0,-0.5)$);
\draw[line width=10pt, color=gray!50, line cap=round] (s_2) -- ($(s_2)+(0,-0.5)$);
\end{scope}

\end{tikzpicture}
    \caption{Three cycles $C_1$, $C_2$, and $C_3$ at weighted distance at most $\ell_0$ and the construction of the weighted bipartite set in the case when $d_\omega(s_1',s_2') \ge \ell_0$.}
    \label{fig:CaseB1}
\end{figure}

Otherwise, $P'$ has weight less than $\ell_0$ and for an illustration of the following argument we refer to Figure~\ref{fig:CaseB2}.

We investigate the weighted distance from $s_1$ to $s_2$.
Let us first show that $d_\omega(x_1,s_1) \le \ell_0+5$.
Indeed, we have that there are at most two vertices in $C_1\setminus T_2$ and $d_\omega(s_1,C_1\setminus T_2) \le \ell_0$ (by choice of $s_1$ and as $C_3 \cap T_2 = \emptyset$).
Moreover, $d_\omega(x_1,s_2') \ge k-7$ (as $s_2' \in C_2\subseteq T_1^c$), which gives $d_\omega(s_1,s_2') \ge k - \ell_0 - 12$.
As $s_1'$ lies on a lightest path from $s_1$ to $s_2'$ we then get
\[d_\omega(s_1,s_1') = d_\omega(s_1,s_2') - d_\omega(s_2',s_1') \ge k  - \ell_0 - 12- d_\omega(s_2',s_1') \, . \]
Because $\omega(C_1)=2k+1$ and $x_2\not \in T_1$ we have $d_\omega(x_1,s_1') + d_\omega(s_1',x_2)  \le k+8$; moreover it holds $d_\omega(x_1,s_1') + d_\omega(s_1',s_2') \ge d_\omega(x_1,s_2') \ge k-7$. Subtracting these two inequalities gives us $d_\omega(x_2,s_1') \le 15+d_\omega(s_2',s_1')$ (algebraically).

As $C_3 \cap T_2 = \emptyset$ and $d_\omega(s_2,C_3) \le \ell_0$ we have $d_\omega(s_2,x_2) \ge k - \ell_0$. This, with the fact that $s_2'$ lies on a shortest $s_1', s_2$-path gives us
\begin{align*}
  d_\omega(s_2,s_2') &=  d_\omega(s_2,s_1')-d_\omega(s_2',s_1')\\ 
  &\ge d_\omega(s_2,x_2) - d_\omega(s_1',x_2) - d_\omega(s_2',s_1') \\
  &\ge  k- \ell_0-2d_\omega(s_2',s_1')  - 15 \, .
\end{align*}
Putting all these together, including $d_\omega(s_1',s_2')=\omega(P')$ which we are now assuming to be at most $\ell_0$, we obtain
\begin{align*}
\omega(P) &= d_\omega(s_1,s_2) = d_\omega(s_1,s_1') + d_\omega(s_1',s_2') + d_\omega(s_2',s_2) \\
&\ge 2k-2d_\omega(s_2',s_1') - 2 \ell_0 - 27 \ge 2k- 4 \ell_0 - 27 \, .
\end{align*}

Let us now fix $s_3$ in $C_3$ a vertex closest to $s_2$ and let us apply Lemma~\ref{lem:PathInCycle} with $i=16$ to obtain a path $P_3$ in $C_3$ with $s_3$ in its unweighted middle. Without loss of generality (by taking a sub-path in which $s_3$ is still in its unweighted middle) we can assume that $|P_3| \le \tfrac 23 k - 59$.
Note that the distance of $s_3$ to any other vertex in $P_3$ is at most $\tfrac 12 (\tfrac 23 k - 59)+1$ and, therefore, the weighted distance of $s_3$ to any other vertex in $P_3$ is at most $\tfrac53 k - \tfrac 52 57$.
Since $d_\omega(s_2,s_3) \le \ell_0$, we deduce
\begin{align*}
d_\omega(s_1,P_3) &\ge d_\omega(s_1,s_2) - d_\omega(s_2,s_3) - \max_{v \in P_3} \{ d_\omega(v,s_3) \} \\
&\ge 2k- 4 \ell_0 - 27 -\tfrac 53 k + \tfrac 52 57 - \ell_0 \ge \tfrac{2}{3}k - 5 \ell_0 + 2 \, .
\end{align*}

Note that since $P_3\cap T_2=\emptyset$, we have $d_\omega(x_2,P_3) \ge k-7$ and so we have two `antipodal' points on $C_1$ that are both `far' from $P_3$, which we can use to show that $P_1$ and $P_3$ are also `far'.
More precisely, because we showed $d_\omega(x_1,s_1) \le \ell_0+5$ and $k-2 \le d_\omega(x_1,x_2) \le k$, and because $C_1$ is a cycle of minimal odd weight (and therefore given two points in the cycle, the natural sub-path of the cycle between them is a path with minimal weight between them) we get that for any vertex $v$ in $P_1$ we have 
\begin{align*}
  d_\omega(x_2,v) &\le d_\omega(x_2,x_1)-d_\omega(x_1,v)\\
  &\le d_\omega(x_2,x_1) + d_\omega(x_1,s_1) - d_\omega(s_1,v) \\
  &\le k+\ell_0+5-d_\omega(v,s_1)  
\end{align*}
and then also
\begin{align*}
d_\omega(v,P_3) &\ge \max\{ d_\omega(x_2,P_3) - d_\omega(x_2,v) , d_\omega(s_1,P_3) - d_\omega(s_1,v) \} \\
&\ge \max\{ k - 7 - (k-d_\omega(v,s_1)-\ell_0-5) , \tfrac{2}{3}k - 5 \ell_0 +2 -d_\omega(v,s_1) \} \\
&\ge \tfrac13 k - 2 \ell_0 \, .
\end{align*}

Because of our lower bound on $k$, we have $d(P_1, P_2)\ge \tfrac15 (\tfrac13 k - 2 \ell_0) \ge 8$. Moreover, we know $N^{10}[P_j]$ is weighted bipartite for $j=1,3$. Therefore, we can employ Corollary~\ref{Cor:ForGood} with $i=3$ and a shortest path $P''$ between $P_1$ and $P_3$ to obtain that $N^3[P_1 \cup P_3 \cup P''] $ is weighted bipartite.
Moreover,
\begin{align*}
   |N^1[P_1 \cup P'' \cup P_3]| &\ge |N^1[P_1]| + |N^1[P_3]| +|P''| - 4 \\
   &\ge 2 (\tfrac 23 k-59) + \tfrac 15 (\tfrac13 k - 2 \ell_0) - 4 \\
   &\ge \tfrac 45 k + t\, .
\end{align*}
Where the last inequality is given by our lower bound $k \ge 9 \ell_0 = 5490 + 45 t$. This concludes the proof in this case with Lemma~\ref{lemlargeInt} applied to $B=N^2[P_1 \cup P'' \cup P_3]$.

\begin{figure}[htbp]
    \centering
\begin{tikzpicture}[scale=1.3]
\draw[] ($(1,0)+(-180:1.3)$) arc (-180:180:1.3cm);
\draw[] ($(-2,0)+(-180:1.3)$) arc (-180:180:1.3cm);
\draw[] (-0.5,-2.7) ellipse (3.4cm and 1cm);
\node[] () at ($(-2,1)$) {\small{$C_1$}};
\node[] () at ($(1,1)$) {\small{$C_2$}};


\foreach \x/\y/\z in {x_2/30/210, s_1/-90/90, x_1/-120/60}{
\node[circle, draw, inner sep=0pt, minimum width=4pt, color=black, fill=black!20, label={[label distance=-0.1cm]\z:\scriptsize{$\x$}}] (\x) at ($(-2,0)+(\y:1.3)$) {};
}

\foreach \x/\y/\z in {s_2/-90/90}{
\node[circle, draw, inner sep=0pt, minimum width=4pt, color=black, fill=black!20, label={[label distance=-0.1cm]\z:\scriptsize{$\x$}}] (\x) at ($(1,0)+(\y:1.3)$) {};
}

\node[circle, draw, inner sep=0pt, minimum width=4pt, color=black, fill=black!20, label={[label distance=-0.1cm]270:\scriptsize{$s_3$}}] () at ($(s_2)+(0,-0.5)$) {};

\node[] () at ($(-0.7,-1.3)$) {\small{$P''$}};
\node[] () at ($(-0.5,-2)$) {\small{$C_3$}};

\draw ($(-0.5, -2.7) + (95:3.4cm and 1cm)$) -- ($(-2,0)+(-50:1.3)$);

\begin{scope}[on background layer]
\draw[line width=10pt, color=blue!20, line cap=round] ($(-2,0)+(45:1.3)$) arc (45:375:1.3cm);

\draw[line width=10pt, color=blue!20, line cap=round] ($(-0.5, -2.7) + (-150:3.4cm and 1cm)$(P) arc (-150:100:3.4cm and 1cm);

\node[] () at ($(-2.9,0)$) {\textcolor{blue}{\small{$P_1$}}};
\node[] () at ($(2.5,-2.7)$) {\textcolor{blue}{\small{$P_3$}}};
\draw[line width=10pt, color=gray!50, line cap=round] (x_2) -- ($(x_2)+(0.7,0)$);

\draw[line width=10pt, color=gray!50, line cap=round] (s_1) -- ($(s_1)+(0,-0.5)$);
\draw[line width=10pt, color=gray!50, line cap=round] (s_2) -- ($(s_2)+(0,-0.5)$);
\end{scope}

\end{tikzpicture}    
    \caption{Three cycles $C_1$, $C_2$, and $C_3$ at weighted distance at most $\ell_0$ and the construction of the weighted bipartite set in the case when $d_\omega(s_1',s_2')\le \ell_0$.}
    \label{fig:CaseB2}
\end{figure}
\end{proof}


\bibliographystyle{amsplain}
\bibliography{chromatic}

\end{document}